\documentclass[11pt]{article}
\usepackage{epsfig}
\usepackage{amssymb,amsmath,amsthm,amscd}
\usepackage{amsfonts, mathabx}
\usepackage{latexsym}
 \usepackage{bbm,dsfont}

\newtheorem{thm}{Theorem}[section]

\newtheorem{cor}[thm]{Corollary}
\newtheorem{lem}[thm]{Lemma}

\newcommand{\be}{\begin{equation}}
\newcommand{\ee}{\end{equation}}

\newcommand{\Z}{\mathbb{Z}^N}
\newcommand{\R}{\mathbb{R}}
\newcommand{\N}{\mathbb{N}}
\newcommand{\E}{\mathbb{E}}

\def \eps {{ \varepsilon }}

\def \o {{  {\mathcal{O}} }}

\def \calf {{  {\mathcal{F}} }}

\def \cala {{  {\mathcal{A}}  }}

\def \calm {{  {\mathcal{M}}  }}

\def \caln {{  {\mathcal{N}}  }}
\def \calk {{  {\mathcal{K}}  }}
\def \caly {{  {\mathcal{Y}}  }}
\def \calz {{  {\mathcal{Z}}  }}
 \def \call {{  {\mathcal{L}}  }}

\begin{document}

\begin{titlepage}
\title{\bf 
Large Deviation Principles
of Invariant Measures 
 of Stochastic Reaction-Diffusion   Lattice Systems
 }
\vspace{7mm}

\author{
 Bixiang Wang  
\vspace{1mm}\\
Department of Mathematics, New Mexico Institute of Mining and
Technology \vspace{1mm}\\ Socorro,  NM~87801, USA \vspace{1mm}\\
Email: bwang@nmt.edu\vspace{6mm}\\
  }

\date{}
\end{titlepage}

\maketitle

\medskip

\begin{abstract}  
In this paper, we study the
  large deviation principle 
 of   invariant measures of
   stochastic  
reaction-diffusion  lattice
systems   driven by 
multiplicative  noise.
We first show that any limit
of a sequence of invariant measures
of the stochastic system
must be an invariant measure of the
deterministic 
limiting system as noise intensity approaches zero.
We then prove the 
uniform Freidlin-Wentzell 
large deviations  of solution paths over all initial data
and the uniform
Dembo-Zeitouni  large deviations
of solution paths over a compact set of initial data.
We finally establish the large deviations
of invariant measures by combining the idea of
tail-ends estimates and the argument of
weighted spaces.
  \end{abstract}



\baselineskip=1.1\baselineskip

\section{Introduction} 
\setcounter{equation}{0}

In this paper, we investigate the
large deviation principle
(LDP) 
of   invariant measures   
 of  the    stochastic
reaction-diffusion
 lattice system  driven by multiplicative 
 noise defined on the
 $N$-dimensional  integer
 set $\Z$.
 Given   $i\in   \Z$, consider the 
 It\^{o} stochastic system:
 \be\label{intr1}
 d u^\eps_i(t)
 + ( \lambda u^\eps_i(t)
 +  (Au^\eps(t))_i
  +f_i ( u^\eps_i(t) ) ) dt  
  =  g_idt  + \sqrt{\eps}
\sigma_i (u^\eps (t) )dW , \quad t>0,
\ee
 with  initial data
 \be\label{intr2}
 u^\eps_i(0)=u_{0,i},
 \ee
  where 
  $i=(i_1,\ldots, i_N)\in \Z$,
$u=(u_i)_{i\in \Z}$
is an unknown sequence,
  $\lambda>0$  
    and $\eps\in [0,1]$  are 
constants, 
$A$ is the  negative discrete 
$N$-dimensional Laplace operator,  
$g=
 (g_{ i})_{  i\in \Z} 
\in \ell^2$,
  and 
   $ W$  is an $\ell^2$-cylindrical Wiener process
   on   a  complete filtered probability 
space
   $(\Omega, \mathcal{F}, 
   \{ { \mathcal{F}} _t\} _{t\in \R},  P )$
   which satisfies   the usual condition.
   For every $i\in \Z$,  
  $f_i:  \R \to \R$  
 is a nonlinear function with  arbitrary growth rate, and $\sigma=(\sigma_i)_{i\in
 \Z}: \ell^2\to \call_2(\ell^2, \ell^2) $
   is a bounded  Lipschitz
function where 
$\call_2(\ell^2, \ell^2)$ is
the space of 
 Hilbert-Schmidt operators
 from $\ell^2 \to \ell^2$
 with norm $\| \cdot \|_{\call_2(\ell^2, \ell^2)}$.

     Lattice  systems  arise
     from many applications 
     in  physical and biological
      systems
     \cite{bel2,  chua2,
     	 kap1, kee1, kee2}.
     The  traveling solutions and   dynamics 
     of   deterministic   lattice systems   have been 
     investigated  in 
     \cite{afr1, bates1, bates2,
     	bates3, bey1,   cho2, cho3, cho4, elm1, elm2,
     	   han3, kara1, mori1, wan8, zin1}
      and the references therein.
      Recently, the dynamics of 
      random lattice systems
      has been extensively
      studied, see, e.g.,
       \cite{bat1, bat3, 
       	car7, car8, han2}
       for 
      random attractors;
      and   \cite{clw,chen1,chen2,
      	lww1, lww2, wan2019, rwan1, xwan1}
      for 
     invariant measures
     and periodic measures.
     The goal of this paper is
     to prove  the
     LDP of invariant measures
     of   system
     \eqref{intr1}-\eqref{intr2}
     as $\eps \to 0$, which is stated below.
     
     \begin{thm}\label{main}
     If \eqref{f1},
     \eqref{sig1}-\eqref{sig4} and \eqref{lgc}
     given in Sections 2 and 4 are fulfilled,
     then the family $\{\mu^\eps\}$
     of invariant measures of \eqref{intr1}-\eqref{intr2}
     satisfies the LDP in $\ell^2$ as $\eps \to 0$.
      \end{thm}

     The LDP of distributions of {\it solutions }
     has been established for a variety of
     stochastic differential equations  in the literature,
     including the stochastic ordinary 
     differential equations as well as the
     stochastic partial differential equations.
     However,
     the LDP of  {\it invariant measures}
      of stochastic systems
     is not well understood, and so far there
     are only a few publications in this area.
     The reader is referred to
     \cite{fre2} for 
     the  LDP of invariant measures  
      of  finite-dimensional
     stochastic systems, and 
     to
    \cite{brz1, cerr0, cerr2, mar1, mar2, sow1}
    for that of  infinite-dimensional stochastic systems.
    In particular, in the infinite-dimensional case,
    the LDP of invariant measures for additive noise
    has been investigated in \cite{brz1,  cerr2, mar1, mar2},
    and for multiplicative noise in 
 \cite{cerr0, sow1}. The stochastic partial
 differential equations in all these papers are
 defined in a bounded spatial domain where
 the Sobolev embeddings 
  are compact.
    As far as the author is aware, there is
     no result reported in the literature
     on  the LDP of invariant measures
     of infinite-dimensional lattice systems
     like \eqref{intr1}.
     
     Note that the lattice system
     \eqref{intr1} can be considered
     as an infinite-dimensional stochastic
     system defined in $\ell^2$
     which has distinct features from the
     finite-dimensional systems associated
     with stochastic ordinary differential equations
     as well as infinite-dimensional
     systems associated with stochastic
     partial differential  equations
     in bounded domains.
     For example,
     every bounded set in a finite-dimensional
     space is precompact, but that is false
     for the infinite-dimensional space 
     $\ell^2$.
     On the other hand,
     since  the Sobolev embedding
     $H^1(\o) \hookrightarrow
      L^2(\o)$  is  compact
     for   a bounded domain $\o$,
     the solution operator 
     generated by the 
     Laplacian  in $ L^2(\o)$ is compact.
    However, since there is
    no such compact embedding 
    available in the discrete case, the solution operator generated by 
     the discrete Laplacian is no longer
     compact in $\ell^2$.
     
     For ordinary differential equations
     and  partial differential  equations
     in bounded domains, the 
        compactness of 
     solution operators   is a key
    for   proving 
    the tightness of  distributions of 
    a family of solutions,
    the compactness of
    the level sets of  rate functions
    as well as    the LDP of
    solutions and  invariant measures.
    Since the solution operator of the
    lattice system \eqref{intr1} is not compact,
    there is an essential difficulty to 
    obtain  the existence and the LDP
    of invariant measures in this case,
   including 
      proving  the exponential tightness of 
     invariant measures 
    over  compact sets.

     In order   to overcome the   
     non-compactness of solutions in
     $\ell^2$, we will  apply 
    the idea of uniform tail-ends  estimates
    to  prove the tightness of distributions
    of a family of solutions and hence the
    existence of invariant measures 
    of \eqref{intr1} as in \cite{wan2019}.
    The uniform  tail-ends estimates
    (see Lemma \ref{tail})
     indicate that  all the  tails of solutions of \eqref{intr1}
     are uniformly small which along with
     the precompactness of bounded sets in
     a finite-dimensional space  yields
     the tightness of distributions of solutions.

   In order to prove
     the LDP of invariant measures
     of \eqref{intr1} in $\ell^2$, we will first
     establish the uniform
     Freidlin-Wentzell  LDP \cite{fre2}  of solutions
    over  all initial data
         (see Theorem \ref{uldp})
    by   
      the weak convergence method as
     developed in 
     \cite{bud1, bud2, dup1,sal2}.
   We will then  establish 
     the  uniform  
   Dembo-Zeitouni
     LDP  \cite{dem1} of solutions
     with respect to 
       a compact set of initial data. 
     We remark that
     both the 
     uniform
     Freidlin-Wentzell  LDP
     and the   uniform  
   Dembo-Zeitouni
     LDP of solutions are needed for proving the
     LDP of invariant measures
     of \eqref{intr1}.
     
     As proved by Salins in \cite{sal2},
      the uniform 
      Freidlin-Wentzell  LDP 
      and
      the uniform Dembo-Zeitouni LDP
      of solutions 
are      not equivalent
    over a non-compact set of  initial data, and hence
    the 
    uniform Dembo-Zeitouni LDP
    of solutions over a
    non-compact set of  initial data
    does not follow from 
    the uniform 
      Freidlin-Wentzell  LDP as obtained
      in Theorem \ref{uldp}.
      Nevertheless, in light of \cite{sal2},
     we can show that 
      Theorem \ref{uldp}
      does imply the 
 uniform Dembo-Zeitouni LDP
 of \eqref{intr1} 
     over a
    compact set of  initial data
    (see Theorem \ref{uldp_DZ}).
     In order to apply 
Theorem \ref{uldp_DZ} to establish 
the LDP of 
      invariant measures of \eqref{intr1},
     we  must 
     prove  the exponential tightness of 
     of invariant measures  over compact subsets
     in $\ell^2$.
     In the finite-dimensional case,  such  a compact
     set  can be chosen as  a  bounded closed subsets
     of the space.
     However, 
      a bounded closed subset
     is not compact in $\ell^2$.
     To solve this  problem,
     in the present paper, we will apply the 
     idea of 
     weighted  $\ell^2$  space 
     and  construct a compact subset
     of $\ell^2$
     from the  
       weighted space  with appropriate
    weight  
     (see  Lemma \ref{ldpc4}).

       It is worth mentioning that
       the stochastic system \eqref{intr1} 
       is  driven by a {\it multiplicative noise},
        and   the 
    uniform Dembo-Zeitouni LDP
 of  solutions of  \eqref{intr1} 
     over a
    bounded set of  initial data
    {\it remains open}.
    For the stochastic reaction-diffusion equation
    driven by  {\it multiplicative noise}
    in a bounded domain,
    the  
    uniform Dembo-Zeitouni LDP of solutions
    was used in \cite{sow1}
    {\it over bounded  initial data}
    and in \cite{cerr0}
    {\it over all initial data}.
    However, there was no proof  given in both papers.
   It seems that 
   it  is nontrivial to prove 
    the  
    uniform Dembo-Zeitouni LDP of solutions
    for   multiplicative noise  
    over bounded  or entire  initial data, and 
    we do not know   how to establish
    the    
    uniform Dembo-Zeitouni LDP 
    of the 
    reaction-diffusion equation
    with multiplicative noise
    over {\it non-compact } initial data.
    That is why in this paper,  we
    prove  the uniform 
    Dembo-Zeitouni LDP 
    only over a compact set of  initial data
    but not over a bounded set of initial data.
    Because   the 
     uniform 
    Dembo-Zeitouni LDP  of solutions
    is available only over compact initial data,
    we must establish the exponential tightness
    of invariant measures over compact sets
    in order to 
    prove 
    the LDP of invariant measures.
    Note that the uniform  exponential probability estimates
    of solutions  
    over  {\it  bounded }  subsets are  not sufficient
    to prove the LDP of invariant measures
    when the solutions satisfy
     the 
     uniform 
    Dembo-Zeitouni LDP 
      only over {\it compact}   subsets of  initial data.
      Of course, 
      the uniform  exponential probability estimates
    of solutions  
    over  {\it  bounded }  subsets 
    can be used to 
      prove the LDP of invariant measures
    when the solutions satisfy
     the 
     uniform 
    Dembo-Zeitouni LDP 
       over {\it  bounded}   subsets of  initial data.
    Since bounded  subsets of $\ell^2$
    are not precompact, we have to derive
    the exponential estimates of  
    solutions in a weighted $\ell^2$  space,
    and use the compactness of embedding
    from a weighted space to the standard
    $\ell^2$ space to 
    obtain  the exponential tightness of
    invariant measures.
    For that purpose, the diffusion term
    must satisfy certain conditions in a weighted
    $\ell^2$ space, see, e.g., \eqref{sig4}.

    This  paper is organized as follows.
    In the next section,
    we prove the existence of invariant measures
    of \eqref{intr1}-\eqref{intr2}
    under some conditions on the nonlinear
    drift and diffusion terms. We also prove 
    that any limit of a sequence of invariant measures
    of the stochastic system
    must be an invariant measure of the limiting
    system with $\eps =0$.
    In Section 3, we prove the uniform 
     Freidlin-Wentzell  LDP  and the 
  uniform Dembo-Zeitouni LDP
 of  solutions.  Section 4 is devoted to
 the  rate functions of invariant measures,
 energy estimates and the exponential tightness
 of solutions.
 In the last two sections, we prove 
     the LDP lower bound and the LDP
    upper bound
    of invariant measures, respectively, by combining the
    ideas of 
    \cite{mar2, sow1}  with   
     weighted spaces.

\section{Existence and limits of invariant measures} 
\setcounter{equation}{0}

 This section is devoted to the 
 existence and the limiting behavior of invariant
 measures of 
  \eqref{intr1}-\eqref{intr2} under appropriate
  conditions. 
  We first  
  discuss the assumptions on the
  nonlinear terms in \eqref{intr1},
  and then show the existence of invariant
  measures of the system. We finally
  prove that any limit of a sequence of
  invariant measures of the stochastic system
  must be an invariant measure of the limiting
  system with $\eps =0$.

  Let $\ell^2$ be the Hilbert space of
  square-summable sequences of real numbers
  on $\Z$ 
 with   norm $ \| \cdot \|$ and  inner product
 $ (\cdot , \cdot )$.
 For  every $j=1,\ldots, N$,  
define  the operators $A_j, B_j, B^*_j :\ell^2\to \ell^2$
  by, for   $u=(u_i)_{i\in \Z}\in \ell^2$
   and $i=(i_1,i_2,\dots,i_N)\in\Z$,
   $$
   (A_ju)_i= -u_{(i_1,\dots,i_j+1,\dots,
 i_N)}+2u_{(i_1,\dots,i_j,\dots,i_N)}-u_{(i_1,\dots,i_j-1,\dots,i_N)},
   $$ 
   $$
    (B_ju)_i=u_{(i_1,\dots,i_j+1,\dots, i_N)}
-u_{(i_1,\dots,i_j,\dots,i_N)},
$$
and 
   $$
     (B_j^{*}u)_i=u_{(i_1,\dots,i_j-1,\dots, i_N)}-u_{(i_1,\dots,i_j,\dots,i_N)}.
     $$
      The  negative discrete 
 Laplace operator $A:\ell^2\to \ell^2$
    is given by,  for   $u=(u_i)_{i\in \Z}\in \ell^2$
   and $i=(i_1,i_2,\dots,i_N)\in\Z$,
$$
    (Au)_i=-u_{(i_1-1,i_2,\dots,i_N)}
 -u_{(i_1,i_2-1,\dots,i_N)} -\dots-u_{(i_1,i_2,\dots,i_N-1)}
  $$
  $$
  +2Nu_{(i_1,i_2,\dots,i_N)}
-u_{(i_1+1,i_2,\dots,i_N)}-u_{(i_1,i_2+1,\dots,i_N)}
 -\dots-u_{(i_1,i_2,\dots,i_N+1)}\big).
$$
Then we have $
A=\sum_{j=1}^N A_j 
 =\sum_{j=1}^N    B_jB_j^{*}
 =\sum_{j=1}^N
 B_j^{*}B_j$.

     For the nonlinear drift  
     $f=(f_i)_{i\in \Z}$,  
we assume  that for every
 $i\in \Z$,
$f_i \in C^1(\R, \R)$ such that 
for $s\in \R$,
\be\label{f1}
f_i ( 
0) =0,
\quad f_i (s)s \ge 0,   \quad \text{and} 
\quad    f_i ^\prime  (  s) 
   \ge - \gamma,
\ee 
where $\gamma>0$ is a constant.

%
 
In the sequel, we assume that
    $\sigma: \ell^2\to \call_2(\ell^2,
 \ell^2)$ 
 is
 a bounded 
 Lipschitz map;   
 that is,  there exists  $L_\sigma>0$
 such that
\be\label{sig1}
 \sup_{u\in \ell^2}
 \|\sigma (u ) 
 \|_{\call_2 (\ell^2, \ell^2)}
 \le  L_\sigma   ,
\ee
and
   \be\label{sig2}
 \|\sigma (u_1)-\sigma (u_2)
 \|_{\call_2 (\ell^2, \ell^2)}
 \le L_\sigma \| u_1-u_2\|,
 \quad
 \forall \  u_1, u_2 \in \ell^2.
\ee
For the existence of invariant measures, we also assume
   \be\label{sig3}
 \lim_{k \to \infty}
 \sup_{u\in \ell^2 }
 \|1_{\{|\cdot |\ge k\}} \sigma (u)\|
 _{\call_2(\ell^2, \ell^2)}
 =0,
\ee
 where 
 $ 
  1_{\{|\cdot |\ge k\}} \sigma (u)  
 $  is  the tail of 
 $\sigma (u)$ as 
 defined by, for every $v\in \ell^2$  
 and $i\in \Z$, 
  $$
   \left (  (1_{\{|\cdot |\ge k\}} \sigma (u) ) (v)
  \right )_i
  =
    1_{[k, \infty) } (|i| )  (\sigma (u)  (v))_i
   =
  \left \{  
  \begin{array}{ll}
    \left (\sigma (u)  (v)
  \right )_i ,& \  \ \text{if }\ |i| \ge k\\
 0,  & \  \ \text{if }\ |i| < k.
 \end{array}
  \right .
  $$
  
  We remark that condition
  \eqref{sig3} is needed when deriving 
  the uniform tail-ends estimates of 
  solutions (see Lemmas \ref{tail} and  \ref{cest2}), which are necessary
  for proving the tightness of distribution laws
  of a family of solutions and hence the existence of
  invariant measures.
  
  We now  introduce a weight function
   ${\kappa}: \R  \to \R$  
      as given by
        $$    {\kappa}  (s) =
     \left ( 1 + s^2
     \right )^{\frac 12}
     , \quad \forall
     \  s\in \R .
     $$ 
   Let  $ \ell^2_{\kappa}
     $ be the weighted space   defined by 
     $$
     l ^2_{\kappa} 
     =\left \{
     u\in \ell^2: \  \| u \|_{\ell^2_{\kappa} }
     =
     \left ( \sum_{i\in \Z}
     {\kappa}^2 (|i|)  | u_i |^2  
     \right )^{\frac 12} <\infty
     \right \}.
     $$ 
   For the LDP of invariant measures,
  we further assume  
  $\sigma: \ell^2 \to \call_2(\ell^2, \ell^2_\kappa)$
  is bounded:
  \be\label{sig4}
   \sup_{u\in \ell^2} 
 \|    \sigma (u)\|
 _{\call_2(\ell^2, \ell^2_\kappa )}
 \le L_\sigma.
\ee

  Since bounded  subsets of $\ell^2$
    are not precompact,  we need to use
    condition \eqref{sig4} to derive 
    the exponential estimates of  
    solutions in  the  weighted space  $\ell^2_\kappa$,
    and  then use the compactness of embedding
    $\ell^2_\kappa \hookrightarrow 
     \ell^2$  to 
    obtain  the exponential tightness of
    invariant measures in $\ell^2$
     (see Lemmas \ref{ldpc3} and \ref{ldpc4}),
  which is necessary to  establish the LDP
  of invariant measures.

  Next, we discuss an example  where 
  $\sigma$ satisfies  \eqref{sig1}-\eqref{sig4}.
  Given $i\in \Z$, let $e_i$ be the element
  in $\ell^2$  which has an  entry 
  $1$ at position $i$;
     and  $0$ otherwise. Then the sequence
   $\{e_i\}_{i\in \Z}$ is an orthonormal basis of
   $\ell^2$. Let $\{a_i\}_{i\in \Z}$ be an
   element in $  \ell^2_\kappa$; that is,
      \be\label{sig0ex}  
      \|a\|^2_{\ell^2_\kappa} =
   \sum_{i\in \Z} (1+|i|^2) a_i^2 <\infty.
  \ee
   Given
   $u= (u_i)_{i\in \Z}$, 
     $v=(v_i)_{i\in \Z}\in \ell^2$, denote by
   \be\label{sig0exa}
    \sigma (u) (v) = 
   (1+ \| u \|)^{-1}
   \sum_{i\in \Z} 
   a_i v_i e_i.
  \ee
   Then $\sigma (u):
   \ell^2 \to \ell^2$ is a Hilbert-Schmidt operator 
   with norm
         \be\label{sig1ex}  
   \|\sigma (u) \|_{\call_2
   (\ell^2, \ell^2)}
   = \| a \| (1+ \|u \|)^{-1}
   \le \|a\|,
   \quad
   \forall \ u\in \ell^2.
   \ee
   In addition, we have
   for all  $u, u_1, u_2 \in \ell^2$,
         \be\label{sig2ex}  
   \|\sigma (u_1) -\sigma (u_2)  \|_{\call_2
   (\ell^2, \ell^2)}
   = \| a \| \| u_1 -u_2\|,
   \ee
     \be\label{sig3ex}  
 \|1_{\{|\cdot |\ge k\}} \sigma (u)\|^2
 _{\call_2(\ell^2, \ell^2)}
 = 
 (1+ \|u \|)^{-2}
 \sum_{|i|\ge k}
 a_i^2 \le  \sum_{|i|\ge k}
 a_i^2 ,
 \ee
 and
   \be\label{sig4ex} 
 \|    \sigma (u)\|^2
 _{\call_2(\ell^2, \ell^2_\kappa )}
 = (1+ \|u \|)^{-2}
  \|a\|^2_{\ell^2_\kappa}
 \le \|a\|^2_{\ell^2_\kappa}.
\ee
It follows from \eqref{sig0ex}-\eqref{sig4ex}
that all conditions
\eqref{sig1}-\eqref{sig4} are satisfied
for $\sigma$ given by
\eqref{sig0exa}
with $L_\sigma =  \|a\|_{\ell^2_\kappa}$.

 We now rewrite 
    system \eqref{intr1}-\eqref{intr2}
      as
the   stochastic equation  in $\ell^2$:
 \be\label{intr3}
 d u^\eps (t)
 +\lambda   u^\eps (t) dt
 + \ Au^\eps (t)  dt
 +f(u^\eps(t) ) dt    =
 g+
   \sqrt{\eps}
   \sigma (u^\eps (t)) dW,
   \quad  u^\eps(0)=u_0\in\ell^2.
\ee

For $\eps=0$,
the deterministic  limiting system corresponding 
to \eqref{intr3} is given by 
 \be\label{intr4}
 {\frac { d u(t)}{dt}}
 +\lambda   u (t)  
 + \ Au (t)   
 +f(u(t) )     =
 g , \quad u(0)=u_0\in\ell^2.
\ee

   Given $\eps\in [0,1]$
 and 
  $u_0 \in L^2(\Omega,\calf_0; 
   \ell^2) $, 
   by a solution $u^\eps$ of \eqref{intr3}, we mean
   that 
   $u^\eps$ is 
   a
  continuous  $\ell^2$-valued  $\calf_t$-adapted
  process    such that
  $u^\eps\in L^2(\Omega, C([0, T], \ell^2))$  for each $T>0$, 
  and for all $t\ge 0$,
   $$
  u^\eps (t)
  +  \int_0^t 
 \left ( \lambda u^\eps (s) + Au^\eps (s) 
 +f(u^\eps(s) )       \right ) ds
 =u_0  + gt  
 +
 \sqrt{\eps}\int_0^t 
   \sigma (u^\eps (s)) dW(s) 
$$ 
in $\ell^2$,
$P$-almost surely.

If \eqref{f1}  and \eqref{sig1}-\eqref{sig2}
are fulfilled, then by the argument of 
   \cite[Theorem 2.5]{wan2019}, one can verify
   that 
   for each $u_0 \in L^2(\Omega,\calf_0; 
   \ell^2) $,
   the stochastic equation \eqref{intr3}
   has a unique solution  
   $u^\eps\in L^2(\Omega, C([0, T], \ell^2))$.
   
   Next, we derive the uniform estimates
   of solutions of \eqref{intr3}.
   
   \begin{lem}\label{est1}
 If \eqref{f1} and
   \eqref{sig1}-\eqref{sig2} are  fulfilled,  
   then  the solution $u^\eps(\cdot, 0,u_0)$  of \eqref{intr3}
  satisfies,  for all  $\eps\in [0,1]$
  and $t\ge 0$,
  $$
  \E ( \| u^\eps (t,0, u_0) \|^2)
  \le  e^{-\lambda t}
  \E(\|u_0\|^2)
  + \lambda^{-2} \|g\|^2
  + \lambda^{-1} L_\sigma^2.
  $$ 
 \end{lem}
 
   \begin{proof}
   By \eqref{f1} and  It\^{o}'s  formula,
   we obtain from  \eqref{intr3}  that
   $$
   {\frac d{dt}}
   \E
   \left (
   \| u^\eps (t)\|^2
   \right )
    + 2 \lambda  \E
   \left (
   \| u^\eps (t)\|^2
   \right )
   \le
   2 \E
   \left (
    u^\eps (t), g 
   \right )
   + \eps
    \E
   \left (
   \| \sigma (u^\eps (t)) \|^2_{\call_2(\ell^2,\ell^2)}
   \right ),
   $$
  which along with \eqref{sig1} implies that
    $$
   {\frac d{dt}}
   \E
   \left (
   \| u^\eps (t)\|^2
   \right )
    +   \lambda  \E
   \left (
   \| u^\eps (t)\|^2
   \right )
   \le \lambda^{-1} \|g \|^2
   +\eps L_\sigma^2.
   $$
   This together with the Gronwall inequality
   concludes the proof. 
 \end{proof}

 The next lemma is concerned with
 the uniform estimates on the tails of solutions.

 \begin{lem}\label{tail}
 Suppose \eqref{f1} and
   \eqref{sig1}-\eqref{sig3} are  valid.  
  Then for any
  $\delta>0$  and 
  any compact subset $\calk$ of $\ell^2$, there exists  
  $k_0=k_0(\delta, \calk) \ge 1$ such that for all $k\ge k_0$,
  $\eps \in [0,1]$
  and $t\ge 0$, the 
   solution $u^\eps (\cdot, 0, u_0)$  of \eqref{intr3}
   with $u_0 \in \calk$
  satisfies,   
  $$
 \sum_{|i| \ge k}  \E \left (
 | u^\eps_i(t, 0, u_0)|^2
 \right )  <\delta .
 $$ 
 \end{lem}

 \begin{proof}
 Let $\rho: \R \to [0,1]$  be a smooth function such that 
 \be\label{cutoff}
 \rho (s)=0 \ \ \text{ if } \  |s| \le 1;    \   \ 
 \rho (s) = 1 \  \  
 \text{ if } \  |s| \ge 2.
\ee
 Given $k\in \N$,  we write 
 $\rho_k u
 =\left (\rho(\frac {|i|}k) u_i
 \right )_{i\in \Z}$
 for any  $u=(u_i)_{i\in \Z}$. 
 By \eqref{f1}, \eqref{intr3} and the It\^{o} formula
 we get 
 \be\label{tail p1}
 {\frac {d}{dt}}
 \E (\|\rho_k u^\eps (t) \|^2)
 + 2 
 \sum_{j=1}^N  \E ( B_j
  u^\eps (t),  \  B_j (\rho_k^2 u^\eps (t) ) )
   +2 \lambda  \E( \| \rho_k u^\eps (t) \|^2 )
 $$
 $$
 \le  2 \E (\rho_k g, \rho_k  u^\eps (t)) 
 +  \eps 
 \E ( \|  \rho_k \sigma (u^\eps (t)) \|^2_{\call_2(\ell^2,
 \ell^2)}  ) .
 \ee
 After simple calculations,  we find
 that for all $t\ge 0$, 
  \be\label{tail p2}
-2 
 \sum_{j=1}^N  \E ( B_j
  u^\eps (t),  \  B_j (\rho_k^2 u^\eps (t) ) )
  \le {\frac {c_1}k} 
    \E ( \|u^\eps (t) \|^2) 
    \ee
    where $c_1$ is 
    a positive number independent of $k$.
     By \eqref{tail p2}
     and  Lemma \ref{est1} we get
     for all $t\ge 0$,
    \be\label{tail p2a}
-2 
 \sum_{j=1}^N  \E ( B_j
  u^\eps (t),  \  B_j (\rho_k^2 u^\eps (t) ) ) 
    \le {\frac {c_2}k},
    \ee
    where $c_2=c_2(\calk)>0$ is independent of $k$.

        By \eqref{sig3} we have
        for all $\eps\in [0,1]$ and $t\ge 0$,
     \be\label{tail p3}
      \eps 
 \E ( \|  \rho_k \sigma (u^\eps (t)) \|^2_{\call_2(\ell^2,
 \ell^2)}  )
 \le
 \sup_{u\in \ell^2}
 \|1_{\{|\cdot|\ge k\}}
 \sigma (u)\|^2_{\call_2(\ell^2,\ell^2)}
 \to 0,
 \ \ \text{as } \   k\to \infty.
\ee
 On the other hand, by Young's inequality we have
   \be\label{tail p4}
  2 \E (\rho_k g, \rho_k  u^\eps (t)) 
 \le \lambda 
 \E (\| \rho_k u^\eps (t)\|^2)
 + {\frac {1}{\lambda}}
  \sum_{|i|\ge k} 
  g_i^2 .
\ee  

It follows from \eqref{tail p1}
and \eqref{tail p2a}-\eqref{tail p4}
that 
for every $\delta>0$,
there exists $k_0=k_0 (\delta, \calk, g)\ge 1$
such that for all $k\ge k_0$ and $t\ge 0$,
  \be\label{tail p5}  
 {\frac {d}{dt}}
 \E (\|\rho_k  u^\eps (t) \|^2)
 +   \lambda  \E( \| \rho_k  u^\eps (t) \|^2 )
  \le  \delta.
 \ee
  By \eqref{tail p5}  and Gronwall\rq{}s inequality, we infer
 that
 for all $k \ge k_0$      and $t\ge 0$,
 \be\label{tail p6}
  \E (\|\rho_k u^\eps (t) \|^2)
  \le e^{- \lambda t}  \E (\|\rho_k u_0  \|^2)
  + {\frac {\delta}{\lambda}} 
   \le  \sum_{|i| \ge k} 
   |  u_{0,i}  |^2 
  + {\frac {\delta}{\lambda}}.
  \ee
  By assumption,  we know that
  $u_0 \in \calk$  and $\calk$ is   compact
  in  $\ell^2$. Therefore,  by \eqref{tail p6} we find
  that  there
  exists $k_1=k_1(\delta, \calk, g) \ge k_0$
  such that   for all $k \ge k_1$, $t\ge 0$
  and $u_0\in \calk$,
  $$
 \E \left (
 \sum_{|i| \ge  2k} |u^\eps_i(t, 0,u_0) |^2 \right )
 \le 
  \E (\|\rho_k u^\eps   (t, 0,u_0) \|^2)
  \le \delta
  + {\frac {\delta}{\lambda}},
 $$
 as desired. 
 \end{proof}
 
 By the argument of Lemma \ref{tail}, we also obtain
 the following uniform tail-ends estimates.

 \begin{lem}\label{taila}
 Suppose \eqref{f1} and
   \eqref{sig1}-\eqref{sig3} are  valid.  
  Then for every
  $\delta>0$  and $R>0$,   there exist
  $T=T(\delta, R)>0$ and  
  $k_0=k_0(\delta) \ge 1$ such that for all 
  $t\ge T$,  $k\ge k_0$ and
  $\eps \in (0,1)$, 
   the 
   solution $u^\eps (\cdot, 0, u_0)$  of \eqref{intr3}
   with $\|u_0\|_{L^2(\Omega, \ell^2)} \le R  $
  satisfies,   
  $$
 \sum_{|i| \ge k}  \E \left (
 | u^\eps_i(t, 0, u_0)|^2
 \right )  <\delta .
 $$ 
 \end{lem}
 
 \begin{proof}
 By \eqref{tail p2}
 and Lemma \ref{est1} we see
 that
 for    $u_0 $
 with $\|u_0\|_{L^2(\Omega, \ell^2)}
 \le R$ and $t\ge 0$, 
  $$
-2 
 \sum_{j=1}^N  \E ( B_j
  u^\eps (t),  \  B_j (\rho_k^2 u^\eps (t) ) )
  \le {\frac {c_1}k}  
  \left (
  e^{-\lambda t}
  R
  + \lambda^{-2} \|g\|^2
  +\lambda^{-1} L_\sigma^2
  \right ),
  $$
  and hence there exists $T_1=T_1(R)>0$
  such that for all $t\ge T_1$,
  \be\label{taila p1}
-2 
 \sum_{j=1}^N  \E ( B_j
  u^\eps (t),  \  B_j (\rho_k^2 u^\eps (t) ) )
  \le {\frac {c_1}k}  
  \left (
  1
  + \lambda^{-2} \|g\|^2
  +\lambda^{-1} L_\sigma^2
  \right ).
    \ee
   By    
  \eqref{tail p3}-\eqref{tail p4}
  and  \eqref{taila p1} we infer 
  from \eqref{tail p1}
  that
 for every $\delta>0$
 and $R>0$,
there exist
$T_2=T_2(\delta, R) \ge T_1$
and $k_0=k_0 (\delta, g)\ge 1$
such that for all $t\ge T_2$ and
$k\ge k_0$,
  \be\label{taila p2}  
 {\frac {d}{dt}}
 \E (\|\rho_k  u^\eps (t) \|^2)
 +   \lambda  \E( \| \rho_k  u^\eps (t) \|^2 )
  \le  \delta.
 \ee
 Then the desired inequality follows from
   \eqref{taila p2}  and Gronwall\rq{}s lemma.
  \end{proof}
 
 We now present the existence of invariant measures
 for \eqref{intr3}.
 
  \begin{thm}\label{exim}
  If   \eqref{f1} and
   \eqref{sig1}-\eqref{sig3} are  fulfilled,  
  then  for every $\eps\in [0,1]$,
  system    \eqref{intr3} 
  has an invariant measure on $\ell^2$.
  \end{thm}
  
  \begin{proof}
  The proof is similar to that
  of \cite[Theorem 4.8]{wan2019},
  by using  the uniform estimates given by
  Lemma \ref{est1} and  the uniform
  tail-ends estimates
  given by Lemma \ref{tail}.
  The details are omitted here.
  \end{proof}

 Given $\eps\in [0,1]$, 
let $\calm^\varepsilon$ be the collection of all invariant
measures of \eqref{intr3}
corresponding to $\eps$.
In the sequel, we will prove the set
 $\bigcup\limits_{\varepsilon\in
[0,1]}\calm ^\varepsilon$ is tight.

\begin{lem} \label{intig}
If   \eqref{f1} and
   \eqref{sig1}-\eqref{sig3} are  fulfilled,  
then  
    $\bigcup\limits_{\varepsilon\in [0,1]} \calm ^\varepsilon$
  is tight.
\end{lem}

\begin{proof}
   By Lemma \ref{taila} we see that
   for every $\delta\in (0,1)$,
  $u_0 \in \ell^2$ and    $k\in \N$,
  there exist
 $T_k=T_k(\delta,  k,  u_0)>0$
 and     $n_k= n_k( \delta, k)\ge 1$ such that
 for all $t\geq T_k$  and $\eps\in  [0,1]$,
 \be\label{intig p1}
 \sum_{|i|\ge n_k}
   \mathbb{E}
              \left( 
              |u^\eps_i (t, 0, u_0)|^2
              \right ) 
    \le \frac{\delta} { 2^{4k} }.
\ee
By Lemma \ref{est1} we know that 
   there exists $T_0 =T_0(u_0)>0$ such that
   for all $t\ge T_0$   and   $\eps \in [0,1]$,
\be\label{intig p2}
   \mathbb{E}
             \left( \| u^{\eps} (t, 0,  u_0)\| ^2
             \right)
   \le c_1,  
\ee
where $c_1= 1+ \lambda^{-2} \|g \|^2
   + \lambda^{-1} L_\sigma^2$.
  For every
  $\delta>0$  and $k\in \N$,
  denote by 
\be \label{intig p3}
    \caly_k^{\delta} 
    = \left\{ 
    u=(u_i)_{i\in \Z} \in \ell^2 :  \
      u_i=0 \ \text{for} \ |i| \geq n_k \
                                      \text{and}\
                          \|    u \|
                          \le  \frac{2^k\sqrt{c_1}}{\sqrt{\delta}}
                               \right\},
\ee
  \be \label{intig p4}
    \calz_k^{\delta}
     =\left\{ u \in \ell^2:\
                             \| u-v \|
                             \le  \frac{1}{2^k}
                           \ \text{for\ some} \
                             v\in \caly_k^{\delta} \right \},
\ee
and
\be \label{intig p5}
    \calk^{\delta}=\bigcap_{k=1}^{ \infty}  
\calz_k^{\delta}.
\ee

By \eqref{intig p3}-\eqref{intig p5} we find that 
  $\calk^\delta$ is compact 
in  $\ell^2$.   Furthermore, by
\eqref{intig p1}-\eqref{intig p5}
and 
the argument of \cite[Theorem 4.4]{chen1},
one can verify that
  that for all $\eps \in [0,1]$ 
 and   $\mu^\eps \in \calm^\eps $,
$$ 
\mu^\eps( \calk^\delta  ) > 1- \delta,
$$
which completes the proof.
\end{proof}
  
    Next, we investigate the limiting behavior
    of a family 
    $\{\mu^\eps\}_{\eps\in (0,1)}$
    of invariant measures of \eqref{intr3}
    as $\eps \to 0$, 
    for which the following convergence of solutions
    is needed.

\begin{lem}\label{csol}
If   \eqref{f1} and
   \eqref{sig1}-\eqref{sig3} are  fulfilled,  
then  for all 
  $t\ge  0$  and  $\delta>0$,
 $$
   \lim_{\eps  \rightarrow  0 }
   \sup_{u_0 \in \ell^2} 
   P \left ( \{\omega\in \Omega :\
    \|
    u^{\eps}(t,0,
    u_0) - u (t, 0, u_0  )\|
     \ge  \delta  \}
   \right )= 0,
 $$
 where $u(\cdot, 0, u_0)$ is the solution
 of \eqref{intr4}.
\end{lem}

    \begin{proof}
    By  \eqref{f1},
    \eqref{intr3}-\eqref{intr4}
    and  the It\^{o} formula we get
    $$
    {\frac d{dt}}
    \E \left (
      \|
    u^{\eps}(t,
    u_0) - u (t,  u_0  )\|^2
    \right )
    \le
    \gamma
    \E \left (
      \|
    u^{\eps}(t,
    u_0) - u (t,  u_0  )\|^2
    \right )
    +\eps
    \E \left (
      \| \sigma (
    u^{\eps}(t,
    u_0) )\|^2_{\call_2(\ell^2,\ell^2)}
    \right ),
    $$
    which along with \eqref{sig1}
    and Gronwall's inequality shows that
    for all $t\ge 0$,
   \be\label{csol p1}
    \sup_{u_0 \in \ell^2} \E \left (
      \|
    u^{\eps}(t,
    u_0) - u (t,  u_0  )\|^2
    \right )
    \le \eps \gamma^{-1} L_\sigma^2
    e^{\gamma t}.
    \ee
    Then for every $\delta>0$, we obtain
    from \eqref{csol p1} that
    $$ 
   \sup_{u_0 \in \ell^2} 
   P \left ( \{\omega\in \Omega :\
    \|
    u^{\eps}(t,
    u_0) - u (t, u_0  )\|
     \ge  \delta  \}
   \right )
   \le
   \eps \delta^{-2}\gamma^{-1} L_\sigma^2
    e^{\gamma t} \to 0,
 $$
 as $\eps \to 0$, as desired.
    \end{proof}

 Next, we show that any 
   limit of a sequence of invariant measures
of the stochastic system  \eqref{intr3} must be
an   invariant measure of the limiting system
\eqref{intr4} 
as $\eps \to 0$.

\begin{thm}
\label{limi}
Suppose  \eqref{f1} and
   \eqref{sig1}-\eqref{sig3} are  
   valid.
   If $\eps_n \rightarrow  0 $
and $\mu^{\eps_n} \in \calm^{\eps_n}$, then
there exists a  subsequence $\eps_{n_k}$ and
an  invariant measure
$\mu^{0} \in \calm^0$ such that
$\mu^{\eps_{n_k}}\rightarrow \mu^{0}$ weakly.
\end{thm}

\begin{proof}
It follows from Lemma \ref{intig}
that
 there exists a subsequence $\eps_{n_k}$ and
 a  probability measure $\mu$
 such that $\mu^{\eps_{n_k}}
 \to  \mu $ weakly, which along with Lemma
 \ref{csol} implies that
 $\mu$ must be an invariant measure
 of the limiting system  \eqref{intr4}
 (see, e.g.,     \cite[Theorem 6.1]{lww1}
 and    \cite[Theorem 2.1]{lche1}).  
 \end{proof}

 Starting from the next section,
 we will  study 
 the LDP of solutions and  invariant measures of \eqref{intr3}
 as $\eps \to 0$.

\section{Uniform large deviations  of solutions} 
\setcounter{equation}{0}

     In this section, we prove the
     uniform   
      Freidlin-Wentzell  LDP of solutions
      as well as 
       the  uniform  
   Dembo-Zeitouni  LDP of solutions
    over all  initial data,
    which are needed for
    establishing the LDP of invariant measures.
    We first  consider the controlled
equation
with control $h\in L^2(0,T; \ell^2)$
 corresponding to \eqref{intr3},
which is given by
 \be\label{ctr1}
 {\frac {d u_h}{dt}}  (t)
+\lambda u_h(t)
   +A   u_h (t)   
  + f(  u_h(t))  
  =      g + \sigma (u_h (t)) h(t), 
   \quad  u_h (0)=u_0\in\ell^2.
\ee

    Notice that
    the limiting system
       \eqref{intr4} is a special case 
       of \eqref{ctr1}
       with
 $h\equiv 0$.

  Given $u_0 \in \ell^2$
and $h\in L^2(0,T; \ell^2)$,  
by a solution  $u_h$  of \eqref{ctr1},
we mean that  
$
  u\in  C([0,T], \ell^2)  
$  and  
 for all $t\in [0,T]$,
$$
  u_h  (t)
+\lambda  \int_0^t u_h(s) ds
   +\int_0^t A   u_h (s)   ds
  +
  \int_0^t  f(  u_h(s))   ds
  =   u_0 +    gt +
  \int_0^t 
   \sigma (u_h (s) h(s) ds  ,  
$$
 in $\ell^2$.
 If \eqref{f1}  and \eqref{sig1}-\eqref{sig2}
 are fulfilled, then one can verify that
 for every $u_0 \in \ell^2$
 and 
 $h\in L^2(0,T; \ell^2)$,
 system \eqref{ctr1} has a unique
 solution $u_h\in C([0,T], \ell^2)$.

   The next lemma is concerned with
   the uniform estimates of solutions
   of \eqref{ctr1}.

     \begin{lem}\label{cest1}
     If \eqref{f1}
     and \eqref{sig1}-\eqref{sig2}
     are fulfilled, then for every
     $u_0\in \ell^2$ 
     and  
     $h\in L^2 (0,T;  \ell^2)$, the solution
      of \eqref{ctr1} satisfies,
       for all $t\in [ 0,T]$,
 $$ 
       \| u_h(t) \|^2
     \le e^{-\lambda t} \|u_0 \|^2
     + 2\lambda^{-2}
     \|g \|^2
     + 2\lambda^{-1} L^2_\sigma
     \int_0^t  \| h(s) \|^2 ds.
     $$
   \end{lem}
     
      \begin{proof}
      By \eqref{f1}  and \eqref{sig1}, we get
      from \eqref{ctr1} that 
      $$
        {\frac d{dt}}
       \| u_h(t) \|^2
       +
       2\lambda \| u_h (t) \|^2
          \le 2  \| u_h (t) \| \| g\|
          + 2\| u_h (t) \| \|\sigma
          (u_h (t))\|_{\call_2(\ell^2,\ell^2)}
          \| h(t) \|
     $$
     $$
     \le
     \lambda  \| u_h (t) \|^2
     +  2\lambda^{-1} \| g \|^2
     + 2\lambda^{-1}   L^2_\sigma \|
     h(t) \|^2
     $$
     which implies that
     for all $t\in [0,T]$,
      $$
       \| u_h(t) \|^2
     \le e^{-\lambda t} \|u_0 \|^2
     + 2\lambda^{-2}
     \|g \|^2
     + 2\lambda^{-1} L^2_\sigma
     \int_0^t
     e^{\lambda (s-t)} \| h(s) \|^2 ds.
     $$
     This completes the proof.
     \end{proof}

               We now  derive the uniform
               estimates on the tails of solutions
               of \eqref{ctr1}. 
   
   \begin{lem}\label{cest2}
    If \eqref{f1}
     and \eqref{sig1}-\eqref{sig3}
     are fulfilled, then  
     for every
   $R>0$,
 $\delta>0$ and    compact subset $\calk$
    of $\ell^2$,
   there exists $k_0=k_0(R, \delta, \calk)
   \in \N$ such that
   for all $k\ge k_0$, 
   $u_0\in \calk$ and  
    $\| h\|_{L^2(0,T; \ell^2)}
   \le R$ with $T>0$,
    the 
   solution $u_h$ of \eqref{ctr1} 
     satisfies 
   $$
  \sup_{0\le t\le T} 
  \sum_{|i|\ge k}
   |\left ( u_h(t,0, u_0)
   \right )_i  |^2 <\delta.
   $$
   \end{lem}
   
    \begin{proof}
    Let $\rho$ be the smooth truncation function
    as given by \eqref{cutoff}.
    By \eqref{f1}  and \eqref{ctr1}  we  get 
 \be\label{cest2 p1}
 {\frac {d}{dt}}
  \|\rho_k u_h  (t) \|^2)
 + 2 
 \sum_{j=1}^N    ( B_j
  u_h  (t),  \  B_j (\rho_k^2 u_h (t) ) )
   +2 \lambda   \| \rho_k u_h  (t) \|^2  
 $$
 $$
 \le  2   (\rho_k g, \rho_k  u_h (t) ) 
 +   2\left (\rho_k  u_h (t), \ \rho_k \sigma (   
 u_h(t)) h(t)  \right ) .
 \ee
 As in \eqref{tail p2a}, by Lemma \ref{cest1} we have
 for all $u_0 \in \calk$ and $\|h\|_{L^2(0,T; \ell^2)}
 \le R$, 
   \be\label{cest2 p2}
-2 
 \sum_{j=1}^N  ( B_j
  u_h  (t),  \  B_j (\rho_k^2 u_h  (t) ) )
  \le {\frac {c_1}k} 
     \|u_h (t) \|^2 
    \le {\frac {c_2}k},
    \ee
     where $c_2=c_2(R,
     \calk)>0$ is
     a constant  independent  of $k$.
     For the last term in \eqref{cest2 p1},
     we have 
$$
     2\left (\rho_k  u_h (t), \ \rho_k \sigma (   
 u_h(t)) h(t)  \right )
 \le
 {\frac 12} \lambda
 \| \rho_k  u_h (t)\|^2
 +2\lambda^{-1} \| \rho_k \sigma (   
 u_h(t) )\|^2_{\call_2(\ell^2, \ell^2)}
 \| h(t)\|^2
 $$
      \be\label{cest2 p3}
 \le
 {\frac 12} \lambda
 \| \rho_k  u_h (t)\|^2
 +2\lambda^{-1}
 \sup_{u\in \ell^2}
 \|1_{\{|\cdot|\ge k\}}
 \sigma (u)\|^2_{\call_2(\ell^2,\ell^2)}
 \| h(t)\|^2.
\ee
 By Young's inequality we have
   \be\label{cest2 p4}
  2   (\rho_k g, \rho_k  u_h  (t)) 
 \le {\frac 12} \lambda 
 \ \| \rho_k u_h  (t)\|^2 
 +  2\lambda^{-1}  
  \sum_{|i|\ge k} 
  g_i^2 .
\ee  

It follows from  \eqref{cest2 p1}-\eqref{cest2 p4}
that  
$$
 {\frac {d}{dt}}
  \|\rho_k  u_h  (t) \|^2 
 +   \lambda   \| \rho_k  u_h (t) \|^2  
 $$
   \be\label{cest2 p5}  
  \le   {\frac {c_2}k}
  +
   2\lambda^{-1}
 \sup_{u\in \ell^2}
 \|1_{\{|\cdot|\ge k\}}
 \sigma (u)\|^2_{\call_2(\ell^2,\ell^2)}
 \| h(t)\|^2
 +  2\lambda^{-1}  
  \sum_{|i|\ge k} 
  g_i^2.
 \ee
 By \eqref{cest2 p5}  and Gronwall\rq{}s inequality, we 
 obtain that for all $t\in [0,T]$,
$$
   \|\rho_k u_h  (t) \|^2
  \le e^{- \lambda t} \|\rho_k u_0  \|^2 
   +   c_2 \lambda^{-1} k^{-1} 
   +  2\lambda^{-2}  
  \sum_{|i|\ge k} 
  g_i^2
  $$
  $$
  +
  2\lambda^{-1} \sup_{u\in \ell^2}
 \|1_{\{|\cdot|\ge k\}}
 \sigma (u)\|^2_{\call_2(\ell^2,\ell^2)}
  \int_0^t
  e^{\lambda (s-t)}
 \| h(s )\|^2 ds
 $$
 \be\label{cest2 p6}
 \le 
 \sum_{|i|\ge k}
 |u_{0,i}|^2
 +    c_2 \lambda^{-1} k^{-1} 
   +  2\lambda^{-2}  
  \sum_{|i|\ge k} 
  g_i^2
   +
  2R^2 \lambda^{-1} \sup_{u\in \ell^2}
 \|1_{\{|\cdot|\ge k\}}
 \sigma (u)\|^2_{\call_2(\ell^2,\ell^2)}.
 \ee
   
  Note that
  $u_0 \in \calk$  and $\calk$ is   compact,
  which along with \eqref{sig3}
  and \eqref{cest2 p6}
   implies that
  for every $R>0$ and 
  $\delta>0$,   
   there
  exists $k_0=k_0(R, \delta, \calk) \in \N$
  such that   for all $k \ge k_0$ and $t\in [0,T] $,
  $$  
  \sum_{|i| \ge  2k} |(u_h (t))_i  |^2 
  \le 
   \|\rho_k u_h  (t) \|^2 <\delta.
  $$
  This completes the proof.
    \end{proof}

 The following uniform tail-ends estimates
 are also useful when studying
 the long term dynamics of 
 the deterministic system \eqref{intr4}.

 \begin{lem}\label{cest2a}
 Suppose \eqref{f1} and
   \eqref{sig1}-\eqref{sig3} are  valid.  
  Then for every
  $\delta>0$  and $R>0$,   there exist
  $T=T(\delta, R)>0$ and  
  $k_0=k_0(\delta) \ge 1$ such that for all 
  $t\ge T$ and  $k\ge k_0$,
   the 
   solution $u(\cdot, 0, u_0)$  of \eqref{intr4} 
   with $\|u_0\| \le R$  
  satisfies  
  $$ 
  \sum_{|i|\ge k}
   | u_i (t,0, u_0)
     |^2 <\delta.
   $$
   \end{lem}
 
 \begin{proof}
 Note that the solution
 $u(\cdot,0, u_0)$
 of \eqref{intr4}
 is also the   solution of \eqref{ctr1}
 with $h \equiv 0$.
 By  \eqref{cest2 p2}
 and   Lemma \ref{cest1} we find
 that
 for   all  
 $\|u_0\| 
 \le R$ and $t\ge 0$, 
  $$
-2 
 \sum_{j=1}^N    ( B_j
  u (t),  \  B_j (\rho_k^2 u (t) ) )
  \le {\frac {c_1}k}  
  \left (
  e^{-\lambda t}
  R
  + \lambda^{-2} \|g\|^2 
  \right ),
  $$
  and hence there exists $T_1=T_1(R)>0$
  such that for all $t\ge T_1$,
  \be\label{cest2a p1}
-2 
 \sum_{j=1}^N    ( B_j
  u  (t),  \  B_j (\rho_k^2 u  (t) ) )
  \le {\frac {c_1}k}  
  \left (
  1
  + \lambda^{-2} \|g\|^2 
  \right ).
    \ee
    Similar to \eqref{cest2 p5}, 
    by \eqref{cest2a p1}
    we find that 
 for every $\delta>0$
 and $R>0$,
there exist
$T_2=T_2(\delta, R) \ge T_1$
and $k_0=k_0 (\delta)\ge 1$
such that for all $t\ge T_2$ and
$k\ge k_0$,
 $$
 {\frac {d}{dt}}
  \|\rho_k  u  (t) \|^2 
 +   \lambda     \| \rho_k  u (t) \|^2  
 < \delta,
 $$
 which implies the desired estimates. 
  \end{proof}

   We now prove the continuity
   of solutions  of \eqref{intr3}
   in initial data.

\begin{lem}\label{cest3}
  If \eqref{f1}
     and \eqref{sig1}-\eqref{sig2}
     are fulfilled, then  
     the solutions $u_{h_i} (\cdot, 0, u_{0,i}) $
of
\eqref{ctr1}  with  
  $i=1,2$,   
satisfy   for all    $t\in [0,T]$, 
\be\label{cest3 1}
\| u_{h_1} (t,  0, u_{0,1} )-u_{h_2} (t,
0,
 u_{0,2}) 
\|^2 
\le
L_1  
\left ( \| u_{0,1}- u_{0,2}\|^2
  + 
 \| h_1  -h_2   \|^2_{L^2(0,T; \ell^2)}
 \right ),  
  \ee
  where $L_1=L_1(T, R)>0$ when
  $\|h_1\|_{L^2(0,T; \ell^2)} \le R$.
  
  In addition, if $\lambda>\gamma$, then we have
    for all $t\in [0,T]$,
    $$
\|  u_{h_1} (t,  0, u_{0,1} )-u_{h_2} (t,
0,
 u_{0,2}) \|^2
 \le
  e^{-\int_0^t (
  \lambda 
-  \gamma 
 -2(\lambda -\gamma)^{-1} L^2_\sigma
 \|h_1 (r)\|^2
  ) dr  } \|  u_{0,1}- u_{0,2} \|^2
  $$
 \be\label{cest3 2}
 +
 2(\lambda -\gamma)^{-1}  L^2_\sigma
 \int_0^t
 e^{\int_t^s (
 \lambda 
-  \gamma 
 -2(\lambda -\gamma)^{-1} L^2_\sigma
 \|h_1 (r)\|^2
 ) dr  }
 \| h_1(s) -h_2 (s) \|^2 ds.
\ee
 \end{lem}

\begin{proof}
Let $v(t) = u_{h_1} (t,0, u_{0,1})
 -u_{h_2}(t,0, u_{0,2})$. 
By \eqref{ctr1} we have
$$ 
{\frac {d}{dt}}
\| v(t) \|^2
+  2 \lambda \| v(t) \|^2
+ 2\sum_{j=1}^N
\| B_j v(t)\|^2
+ 2  \left (f(u_{h_1}(t, u_{0,1}))
-f(u_{h_2} (t,  u_{0,2})),
v(t) \right )
$$
\be\label{cest3 p1}
  =2
  \left (
  \sigma (u_{h_1} (t, u_{0,1}) ) h_1 (t)
  -
    \sigma (u_{h_2} (t,  u_{0,2}) ) h_2(t),
    \ v(t)
  \right ).
\ee
By \eqref{f1} we get
\be\label{cest3 p2}
2  \left (f(u_{h_1}(t, u_{0,1}))
-f(u_{h_2} (t,  u_{0,2})),
v(t) \right )
\ge -2 \gamma \| v(t)\|^2.
\ee
For the last term on the right-hand side of
\eqref{cest3 p1},
by \eqref{sig1}-\eqref{sig2}  we obtain
$$
2
  \left (
  \sigma (u_{h_1} (t, u_{0,1}) ) h_1 (t)
  -
    \sigma (u_{h_2} (t,  u_{0,2}) ) h_2(t),
    \ v(t)
  \right )
  $$
  $$
=2
  \left ( (
  \sigma (u_{h_1} (t, u_{0,1}) ) 
  -  \sigma (u_{h_2} (t, u_{0,2}) ))
  h_1 (t) ,
    \ v(t)
  \right )
  $$
  $$
  + 2
  \left (   
     \sigma (u_{h_2} (t, u_{0,2})  )
  (h_1 (t) -h_2 (t)) ,
    \ v(t)
  \right )
  $$
 \be\label{cest3 p3}
\le 2L_\sigma \|h_1(t) \| \| v(t) \|^2
+ 2L_\sigma \|h_1(t)-h_2(t)\| \| v(t) \|.
 \ee
 By Young's inequality we have
   $$
   2L_\sigma \|h_1(t)-h_2(t)\| \| v(t) \|
   \le \lambda \| v(t) \|^2
   + \lambda^{-1} L^2_\sigma \| h_1 (t) -h_2 (t) \|^2,
   $$
   which together with \eqref{cest3 p3}
   shows that
   $$
2
  \left (
  \sigma (u_{h_1} (t, u_{0,1}) ) h_1 (t)
  -
    \sigma (u_{h_2} (t,  u_{0,2}) ) h_2(t),
    \ v(t)
  \right )
  $$
   \be\label{cest3 p4}
  \le
    (\lambda +
    2L_\sigma \|h_1(t) \|)  \| v(t) \|^2 
    + \lambda^{-1} L^2_\sigma \| h_1 (t) -h_2 (t) \|^2.
    \ee
    
    It follows from \eqref{cest3 p1}-\eqref{cest3 p2}
    and \eqref{cest3 p4} that
    $$ 
{\frac {d}{dt}}
\| v(t) \|^2
-2  \left ( 
  \gamma 
 + L_\sigma
 \|h_1 (t)\|
 \right ) \| v(t) \|^2
 \le 
  \lambda^{-1} L^2_\sigma \| h_1 (t) -h_2 (t) \|^2,
   $$
   and hence for all $t\in [0,T]$,
    $$
\| v(t) \|^2
 \le
  e^{2 \int_0^t ( 
 \gamma 
 + L_\sigma
 \|h_1 (r)\|) dr  } \| v(0)\|^2
  $$
 $$
 +
 \lambda^{-1} L^2_\sigma
 \int_0^t
 e^{2\int^t_s ( 
 \gamma 
 + L_\sigma
 \|h_1 (r)\|) dr  }
 \| h_1(s) -h_2 (s) \|^2 ds,
  $$
which  implies \eqref{cest3 1}.

If $\lambda>\gamma$, then we have 
   $$
   2L_\sigma \|h_1(t)-h_2(t)\| \| v(t) \|
   \le {\frac 12} (\lambda -\gamma)
    \| v(t) \|^2
   + 2(\lambda -\gamma)^{-1}
    L^2_\sigma \| h_1 (t) -h_2 (t) \|^2,
   $$
   which together with \eqref{cest3 p3}
   shows that
   $$
2
  \left (
  \sigma (u_{h_1} (t, u_{0,1}) ) h_1 (t)
  -
    \sigma (u_{h_2} (t,  u_{0,2}) ) h_2(t),
    \ v(t)
  \right )
  $$
   $$
  \le
    ( {\frac 12} (\lambda -\gamma) +
    2L_\sigma \|h_1(t) \|)  \| v(t) \|^2 
    + 2   (\lambda -\gamma)^{-1}
     L^2_\sigma \| h_1 (t) -h_2 (t) \|^2
   $$
       \be\label{cest3 p4a}
  \le
    (   (\lambda -\gamma) +
    2(\lambda -\gamma)^{-1}
    L^2_\sigma \|h_1(t) \|^2 )  \| v(t) \|^2 
    + 2   (\lambda -\gamma)^{-1}
     L^2_\sigma \| h_1 (t) -h_2 (t) \|^2.
    \ee
    It follows from \eqref{cest3 p1}-\eqref{cest3 p2}
    and \eqref{cest3 p4a} that
    $$ 
{\frac {d}{dt}}
\| v(t) \|^2
+  \left (  \lambda 
-  \gamma 
 -2(\lambda -\gamma)^{-1} L^2_\sigma
 \|h_1 (t)\|^2
 \right ) \| v(t) \|^2
 \le 
   2(\lambda -\gamma)^{-1}  L^2_\sigma \| h_1 (t) -h_2 (t) \|^2,
   $$
   and hence for all $t\in [0,T]$,
    $$
\| v(t) \|^2
 \le
  e^{-\int_0^t (
  \lambda 
-  \gamma 
 -2(\lambda -\gamma)^{-1} L^2_\sigma
 \|h_1 (r)\|^2
  ) dr  } \| v(0)\|^2
  $$
 $$
 +
 2(\lambda -\gamma)^{-1}  L^2_\sigma
 \int_0^t
 e^{\int_t^s (
 \lambda 
-  \gamma 
 -2(\lambda -\gamma)^{-1} L^2_\sigma
 \|h_1 (r)\|^2
 ) dr  }
 \| h_1(s) -h_2 (s) \|^2 ds.
  $$
which  implies \eqref{cest3 2} and thus
completes the proof.
  \end{proof}

In terms of \eqref{ctr1}, 
for every
$T>0$ and
$u_0\in \ell^2$, we define a rate function
$I_{T, u_0}: 
  C([0,T], \ell^2) \to [0, +\infty]$  by: 
\be\label{act1}
I_{T, u_0}
(u)  
= \inf \left \{ 
 {\frac 12} \int_{0}^{T}
 \| h (t) \|^2 dt:\
 h\in L^2(0,T; \ell^2) , 
 \  u(0) =u_0,\ \ u =u_h
 \right \},  
\ee
where $u_h$ is the solution of
\eqref{ctr1}.  Note that the infimum
of the empty set is 
  $+\infty$.
  
  By the argument of \cite{wan2024},
  one can verify
  that if $h_n \to h$ weakly in $L^2(0,T; \ell^2)$,
  then $u_{h_n} (\cdot, 0, u_0)
  \to u_h(\cdot, 0, u_0)$ strongly
  in $C([0,T], \ell^2)$ for every
  $u_0 \in \ell^2$, which implies that
for every $T>0$, $N>0$
and $u_0 \in \ell^2$,
the set
\be\label{lset0}
\left \{
u_h (\cdot, 0, u_0):
\|h\|_{   L^2(0,T; \ell^2)}
\le N
\right \}
\ \text{ is a compact set in } \ 
  C([0,T], \ell^2).
\ee 
  For each $ s \ge 0$, the $s$-level set
of $I_{T,u_0} $ is given  by
\be\label{lset1}
I_{T,u_0}^s
=\left \{
u\in C([0,T], \ell^2):
I_{T, u_0} (u) \le s
\right \}.
\ee
 It follows from \eqref{lset0} that
 $I_{T,u_0}^s$ is compact for all
$s\ge 0$, and hence
$I_{T,u_0} $ is a good rate function.

Next,  we will prove the solutions
of the stochastic system
\eqref{intr3} satisfies 
the  
Freidlin-Wentzell
uniform LDP  with  respect to  all
initial data in $\ell^2$, for which
we need to consider the following stochastic system
with $h\in  L^2(0,
T; \ell^2)$:
\be\label{gep}
 d   u_{h}^\eps  
 + \left ( \lambda u_{h}^\eps 
 +
  Au^\eps_{h}     
 +f(u^\eps_{h}  )  
\right )  dt  =
\left (
g 
+\sigma(u^\eps_{h}  )  h
\right )  dt
 + \sqrt{\eps} \sigma( u^\eps_{h}  )  
   dW   ,
\ee
 with  initial condition 
 $
 u^\eps_{h} (0)=u_0\in\ell^2. $

 Note that system \eqref{gep} is
 obtained from \eqref{intr3}
 by replacing
 $W$ by 
 $W +\eps^{-\frac 12}\int_0^\cdot
h(t) dt$, which is also an
$\ell^2$-cylindrical Wiener process by
Girsanov's theorem.
Consequently, the existence and uniqueness
of solutions
of \eqref{gep} follows from
that of system \eqref{intr3}.

%

If $\eps =0$, then 
system \eqref{gep} reduces to the
control system \eqref{ctr1}.
The relation between
the solutions of  \eqref{ctr1}
and \eqref{gep}  is given below.

 \begin{lem}\label{cgep}
  If 
 \eqref{f1} and  \eqref{sig1}-\eqref{sig2}
 hold,  then 
 for every $N>0$ and $T>0$, 
 the solutions   of \eqref{ctr1}
 and \eqref{gep} satisfy
 $$
 \lim_{\eps \to 0}
 \sup_{u_0\in \ell^2}
 \sup_{h\in \cala_N}\E \left (  
 \|  u_{h}^\eps (\cdot, 0, u_0)
  -  u_{h}(\cdot,0, u_0) \|^2 _{C([0,T], \ell^2)}
  \right ) =0,
  $$
  where $\cala_N$ 
  is the set of all $\ell^2$-valued processes
  $h$ which are progressively measurable
  and $  \| h \|_{L^2(0,T; \ell^2)} \le N$,
  $P$-almost surely.
  \end{lem}

\begin{proof}
  By   
\eqref{ctr1} and \eqref{gep} we have 
   $$
 d  ( u_{h}^\eps -  u_{h}  )
 +
 + \lambda  (  u^\eps_{h}  
-  u_{h } ) dt
 +     A ( u^\eps_{h}  -  u_{h  }
 ) dt 
 +
   ( f(u^\eps_{h}  ) -  f(u_{h}  ) ) dt
$$
  \be\label{cgep p1}
  =
\left ( 
 \sigma(u^\eps_{h}  )  h
 -
  \sigma(u_{h}  )  h
\right )  dt
 + \sqrt{\eps} \sigma( u^\eps_{h}  )  
   dW  ,
\ee
with  
$u_{h}^\eps(0) -  u_{h} (0)=0$.
 By 
 \eqref{f1},
 \eqref{cgep p1} and It\^{o}\rq{}s formula we obtain
$$
 \|  u_{h}^\eps(t) -  u_{h}(t) \|^2
 +
   2  (\lambda -\gamma)
  \int_0^t
  \| u^\eps_{h}(s)  -  u_{h}(s)
  \|^2  ds
  $$
 $$
 \le 
  2 
  \int_0^t
  \left ( 
 \sigma( u^\eps_{h} (s) )  h(s)
 -
  \sigma( u_{h} (s)  )  h (s),
  \  u^\eps_{h}(s)  -  u_{h}(s)
\right )  dt
$$
  $$
+ \eps
\int_0^t
 \|   \sigma( u^\eps_{h} (s)  )\|
_{\call_2(H, \ell^2)}^2  ds
+ 
2\sqrt{\eps} \int_0^t
\left (
u^\eps_{h}(s)  -  u_{h}(s),\
 \sigma( u^\eps_{h} (s)  ) 
   dW  \right ) ,
   $$ and hence
   for all $t\in [0,T]$,
 $$ 
 \sup_{0\le r \le t}
 \|  u_{h}^\eps(r )
  -  u_{h}( r ) \|^2
 \le   2  \gamma
  \int_0^{t  }
  \| u^\eps_{h}(s)  -  u_{h}(s)
  \|^2  ds
  $$
  $$
  + 
  2  
  \int_0^{ t }
  \|
 \sigma( u^\eps_{h} (s)  )  
 -
  \sigma( u_{h} (s)  )\|_{\call_2
  (\ell^2, \ell^2)} \|  h (s)\|  \|
    u^\eps_{h}(s)  -  u_{h}(s)\|
   dt
$$
   \be\label{cgep p2}
+ \eps
\int_0^ { t }
 \|   \sigma( u^\eps_{h} (s)  )\|
_{\call_2(H, \ell^2)}^2  ds
+ 
2\sqrt{\eps}
\sup_{0\le r\le t}
 \left |\int_0^ { r }
\left (
u^\eps_{h}(s)  -  u_{h}(s),\
 \sigma( u^\eps_{h} (s)  ) 
   dW  \right )
   \right | .
\ee
By \eqref{sig2} we have
$$
   2  
  \int_0^{ t }
  \|
 \sigma( u^\eps_{h} (s)  )  
 -
  \sigma( u_{h} (s)  )\|_{\call_2
  (\ell^2, \ell^2)} \|  h (s)\|  \|
    u^\eps_{h}(s)  -  u_{h}(s)\|
   dt
   $$
      \be\label{cgep p3}
\le
2L_\sigma \int_0^t
 \|  h (s)\| 
 \sup_{0\le r \le s}
  \|
    u^\eps_{h}(r)  -  u_{h}(r)\|^2 ds.
    \ee
By \eqref{sig1} we get
for all  $t\in [0,T]$, $\eps \in [0,1]$
and $h\in \cala_N$,
 \be\label{cgep p4}
 \eps
\int_0^ { t }
 \|   \sigma( u^\eps_{h} (s)  )\|
_{\call_2(H, \ell^2)}^2  ds
\le \eps L_\sigma^2T.
\ee
       By \eqref{cgep p2}-\eqref{cgep p4}
        we get
for all  $t\in [0,T]$, $\eps \in [0,1]$
and $h\in \cala_N$,
$$
  \sup_{0\le r \le t}
 \|  u_{h}^\eps(r )
  -  u_{h}( r ) \|^2
 \le    
  \int_0^{t  }
  \left (
  2  \gamma
  +2L_\sigma \| h(s) \|  \right ) 
  \sup_{0\le r\le s} \|   u_{h}^\eps(r )
  -  u_{h}( r ) \|^2
  ds
 $$
    \be\label{cgep p5}
    +
\eps L_\sigma^2 T  
  + 
2\sqrt{\eps}
\sup_{0\le r\le T}
 \left |\int_0^ { r }
\left (
u^\eps_{h}(s)  -  u_{h}(s),\
 \sigma( u^\eps_{h} (s)  ) 
   dW  \right )
   \right | .
\ee
   By \eqref{cgep p5} and Gronwall\rq{}s lemma,
  we get
  for all  $t\in [0,T]$, $\eps \in [0,1]$
and $h\in \cala_N$, 
$P$-almost surely,
  \be\label{cgep p6}
 \sup_{0\le r \le t}
 \|  u_{h}^\eps(r)
  -  u_{h}( r) \|^2 
\le 
\eps L_\sigma^2 Tc_1  
  + 
2\sqrt{\eps}c_1
\sup_{0\le r\le T}
 \left |\int_0^ { r }
\left (
u^\eps_{h}(s)  -  u_{h}(s),\
 \sigma( u^\eps_{h} (s)  ) 
   dW  \right )
   \right |
 \ee
where 
$ 
c_1= 
e^{
 2\gamma T +
2L_\sigma  T^{\frac 12} N  
 } $.  
It follows from \eqref{cgep p6}   that
for all    $\eps \in [0,1]$
and $h\in \cala_N$,  
$$
 \E \left ( \sup_{0\le r \le T}
 \|  u_{h}^\eps(r)
  -  u_{h}( r) \|^2 
  \right )
  $$
    \be\label{cgep p7}
\le 
\eps L_\sigma^2 Tc_1  
  + 
2\sqrt{\eps}c_1
\E \left (
\sup_{0\le r\le T}
 \left |\int_0^ { r }
\left (
u^\eps_{h}(s)  -  u_{h}(s),\
 \sigma( u^\eps_{h} (s)  ) 
   dW  \right )
   \right |
   \right ).
 \ee
 By \eqref{sig1}  and
 the Burkholder inequality we have
 $$
2\sqrt{\eps}c_1
\E \left (
\sup_{0\le r\le T}
 \left |\int_0^ { r }
\left (
u^\eps_{h}(s)  -  u_{h}(s),\
 \sigma( u^\eps_{h} (s)  ) 
   dW  \right )
   \right |
   \right )
$$
$$
\le 6 \sqrt{\eps}c_1
\E \left ( 
\left (
 \int_0^ {T}
 \|
u^\eps_{h}(s)  -  u_{h}(s)\|^2
 \|\sigma( u^\eps_{h} (s)  ) \|^2_{\call_2
 (\ell^2, \ell^2)} ds
     \right )^{\frac 12}
     \right )
 $$
 $$
\le 6 \sqrt{\eps}c_1
\E \left ( 
\sup_{0\le r \le T}
\|
u^\eps_{h}(r)  -  u_{h}(r)\| 
\left (
 \int_0^ {T}
  \|\sigma( u^\eps_{h} (s)  ) \|^2_{\call_2
 (\ell^2, \ell^2)} ds
     \right )^{\frac 12}
     \right )
 $$
 $$
\le 
{\frac 12}
\E \left ( 
\sup_{0\le r \le T}
\|
u^\eps_{h}(r)  -  u_{h}(r)\| ^2
\right )
+
18 \eps c_1^2 
\E \left (    
 \int_0^ {T}
  \|\sigma( u^\eps_{h} (s)  ) \|^2_{\call_2
 (\ell^2, \ell^2)} ds 
     \right )
 $$
 $$
\le 
{\frac 12}
\E \left ( 
\sup_{0\le r \le T}
\|
u^\eps_{h}(r)  -  u_{h}(r)\| ^2
\right )
+
18 \eps c_1^2 TL^2_\sigma,
  $$
which along with \eqref{cgep p7} shows that
for all    $\eps \in [0,1]$,
 $$
 \sup_{u_0\in \ell^2}
 \sup_{h\in \cala_N}\E \left ( \sup_{0\le r \le T}
 \|  u_{h}^\eps(r)
  -  u_{h}( r) \|^2 
  \right )
  \le 
2 \eps L_\sigma^2 Tc_1  
  +  
36 \eps c_1^2 TL^2_\sigma.
$$
This completes the proof.
\end{proof}

As an immediate consequence
of Lemma \ref{cgep}, we have
the following convergence of solutions
in probability.

\begin{cor}\label{cgep1}
  If 
 \eqref{f1} and  \eqref{sig1}-\eqref{sig2}
 hold,  then for every $N>0$, 
 $T>0$ and    $\delta>0$,
 the solutions of \eqref{ctr1} and
 \eqref{gep} satisfy
 $$
 \lim_{\eps \to 0}\sup_{u_0 \in \ell^2 } \
 \sup_{h\in \cala_N}
 P
 \left (
 \| u^\eps_h (\cdot, 0, u_0)
 -u_h (\cdot, 0, u_0)\|_{C([0,T], \ell^2)}
 >\delta
 \right )
 =0.
 $$
  \end{cor}

  From now on,   
for every 
$\eps\in [0,1]$,
$T>0$ 
and $u_0\in \ell^2$,
we denote 
by 
$\nu^\eps_{u_0}$
 the distribution law
of the solution 
$
 u^\eps
 (\cdot, 0, u_0)$ 
of \eqref{intr3} 
    in
 $C([0,T], \ell^2)$.
 Then by \eqref{lset0} and 
 Corollary \ref{cgep1} we infer
 from \cite[Theorem 2.13]{sal2} that
 the family
$\{\nu^\eps_{u_0}\}
_{\eps\in (0,1)}$
satisfies the  
Freidlin-Wentzell
uniform LDP
on $C([0,T], \ell^2)$
with rate function $I_{T, u_0}$
  uniformly for all
  $u_0\in \ell^2$.

\begin{thm}\label{uldp}
If    \eqref{f1} and \eqref{sig1}-\eqref{sig2}
hold, then
for every $T>0$,
 the family
$\{u^\eps(\cdot, 0, u_0) \}_{0<\eps< 1}$ of  the solutions
of  \eqref{intr3} 
satisfies  
the  
Freidlin-Wentzell
uniform LDP
on $C([0,T], \ell^2)$
  with rate function $I_{T,u_0}$, 
uniformly for all
$u_0$ in $\ell^2$;  
more precisely:

\begin{enumerate} 
 \item  For every 
    $s\ge 0$,
 $\delta_1>0$  and $\delta_2>0$, there exists
 $\eps_0>0$ such that
\be\label{uldp 1}
 \inf_{ u_0 \in \ell^2} \ 
 \inf_{
 \xi \in I_{T, u_0}^s }
 \left (
 \nu^{\eps}_{u_0} (
 \caln(\xi, \delta_1))
 -e^{- {\frac {I_{T,u_0} 
        (\xi  ) +\delta_2}\eps} }
 \right ) \ge 0,
 \quad \forall \ \eps \le \eps_0  ,
\ee
 where
 $\caln(\xi, \delta_1)
 =\left \{
 u\in C([0,T], \ell^2): \| u-\xi\|
 _{C([0,T], \ell^2)} <\delta_1
 \right \}$.

 \item For every 
 $s_0\ge 0$,
 $\delta_1>0$  and
 $\delta_2>0$,
 there exists $\eps_0>0$ such that
  $$
 \sup_{  u_0 \in \ell^2
  } 
 \nu^{\eps}_{u_0}
 \left (
 {C([0,T], \ell^2 )}
 \setminus \caln(I _{T, u_0}
 ^s, \delta_1)
 \right )
 \le
 e^{-{\frac {s-\delta_2}{\eps}} },
 \quad \forall \ \eps\le \eps_0,
 \ \forall \ s\le s_0,
$$
 where $\caln (I_{T,u_0}^s,\delta_1)
 =\left \{
 u\in  {C([0,T], \ell^2)}: {\rm{dist}} (u, I_{T,u_0}
 ^s) <\delta_1
 \right \} $.
  \end{enumerate}
\end{thm}

  The goal of the present paper
  is to prove  the LDP of the
invariant measures of the stochastic
equation \eqref{intr3},
for which   the
Dembo-Zeitouni uniform
LDP of solutions is also needed.
The   following  continuity of level sets
   of $I_{T,u_0}$  will be used
   when proving 
   the
Dembo-Zeitouni uniform
LDP of solutions.

   \begin{lem}
   \label{levc}
   If \eqref{f1}
   and \eqref{sig1}-\eqref{sig2}
   hold, and 
     $u_{0,n} \to u_0$ in $\ell^2$, then
   for every $T>0$  and $s\ge 0$, 
   $$
   \lim_{n\to \infty}
   \max
   \left \{
   \sup_{v\in I^s_{T,u_0}}
   {\rm dist}_{C([0,T],\ell^2)}
   (v,  I^s_{T, u_{0,n}}),
   \ \ 
    \sup_{v\in I^s_{T,u_{0,n}}}
   {\rm dist}_{C([0,T],\ell^2)}
   (v, I^s_{T, u_{0}})
   \right \} =0. 
   $$
    \end{lem}

  \begin{proof}
    Note that if $v\in I^s_{T,u_0}$,
    then 
    $v= u_h (\cdot, 0, u_0)$
    for some $h$ with
    $\| h\|_{L^2(0,T; \ell^2)}^2\le 2s$.  
 Then  
 $u_h (\cdot, u_{0,n}) \in I^s_{T, u_{0,n}}$
 and   by Lemma \ref{cest3}
 we  have
   $$
   {\rm dist}_{C([0,T],\ell^2)}
   (v, I^s_{T, u_{0,n}})
   \le 
    {\rm dist}_{C([0,T],\ell^2)}
   ( u_h(\cdot, 0, u_{0} )  , \
   u_h(\cdot,  0, u_{0,n} ))
   \le c_1\| u_{0,n} -u_0\|,
  $$
  where $c_1>0$ is a constant independent of $n$,
  and hence 
  \be\label{levc p1}
    \sup_{v\in I^s_{T,u_0}}
   {\rm dist}_{C([0,T],\ell^2)}
   (v, I^s_{T, u_{0,n}})
    \le c_1\| u_{0,n} -u_0\| \to 0.
  \ee
  By the same argument, we can also obtain
  $$ 
   \sup_{v\in I^s_{T,u_{0,n}}}
   {\rm dist}_{C([0,T],\ell^2)}
   (v, I^s_{T, u_{0}})
  \le c_1 \| u_{0,n} -u_0\| \to 0,  
  $$
  which together with \eqref{levc p1}
  concludes the proof.
    \end{proof}

    It follows from
     Lemma \ref{levc} and
  \cite[Theorem 2.7]{sal2}
  that the
  Freidlin-Wentzell and  
  the Dembo-Zeitouni
uniform LDP of \eqref{intr3}
on $C([0,T], \ell^2)$
over a compact subset of $\ell^2$
are equivalent, which along with 
  Theorem \ref{uldp} 
  yields the following result.

\begin{thm}\label{uldp_DZ}
If    \eqref{f1}
and \eqref{sig1}-\eqref{sig2}
hold,   then
 the family
$\{u^\eps(\cdot, 0, u_0) \}_{0<\eps< 1}$ of  the solutions
of   \eqref{intr3}
satisfies  
the   Dembo-Zeitouni
uniform LDP
on $C([0,T], \ell^2)$
  with rate function $I_{T,u_0}$, 
uniformly with respect to 
$u_0$
in all compact subsets  of $\ell^2$;
more precisely:

\begin{enumerate} 
 \item  For any
 compact subset  $\calk $ of $\ell^2$
 and any open subset $G$ of $C([0,T], \ell^2)$,
 $$
 \liminf_{\eps \to 0}
  \inf _{u_0 \in \calk}
  \left (
  \eps \ln 
  P(u^\eps (\cdot, 0,  u_0)
  \in G )
  \right )
  \ge -\sup_{u_0 \in \calk}
  \inf_{v\in G} I_{T, u_0} (v).
 $$
 
 \item  For any
 compact subset  $\calk $ of $\ell^2$
 and 
 any closed subset $F$ of $C([0,T], \ell^2)$,
 $$
 \limsup_{\eps \to 0}
  \sup_{u_0 \in \calk}
  \left (
  \eps \ln 
  P(u^\eps ( \cdot, 0, u_0)
  \in F )
  \right )
   \le  -\inf_{u_0 \in \calk}
  \inf_{v\in F} I_{T, u_0} (v).
 $$
  \end{enumerate}
\end{thm}

\section{Rate functions of  invariant measures} 
\setcounter{equation}{0}

In this section, we introduce a rate function
for the invariant measures of \eqref{intr3}
and prove the exponential tightness of invariant measures,
for which we  first consider the 
  long term dynamics
of the deterministic  limiting system
\eqref{intr4}.

Given $t\ge 0$ and $u_0 \in \ell^2$,  denote by 
$ 
S(t)u_0 =  u(t, 0, u_0)$, 
 which  is the solution of \eqref{intr4}  with initial value $u_0$
at initial time $0$.
  By Lemma  \ref{cest1} with $h\equiv 0$ we see that
  the semigroup $\{S(t)\}_{t\ge 0}$
  has a bounded absorbing set in $\ell^2$.
  On the other hand, by Lemma \ref{cest2a}
  we know that $\{S(t)\}_{t\ge 0}$
   is asymptotically compact in $\ell^2$.
   Consequently,   the dynamical system $\{S(t)\}_{t\ge 0}$
 has a     global attractor
$\cala$ in $\ell^2$ in the  sense that
  $\cala$ is   a  compact 
   invariant set in $\ell^2$ and 
  attracts every bounded subset of $\ell^2$.
  Furthermore, by Lemma \ref{cest1}
  with $h\equiv 0$, we have 
 \be\label{atthbd}
 \sup_{v\in \cala}
 \| v \|^2  \le  2\lambda^{-2}
 \| g \|^2. 
\ee
   If we further assume that
   \be\label{lgc}
   \lambda>\gamma,
   \ee
 then by \eqref{cest3 2} with $h_1\equiv h_2\equiv 0$, we see that
 any two solutions of \eqref{intr4} converge to
 each other as $t\to \infty$, and hence
 $\cala$ is a singleton.
 In this case, we write $\cala=\{u_*\}$.
 By \eqref{atthbd} we know 
  $\| u_* \|^2  \le  2\lambda^{-2}
 \| g \|^2$.

By the idea of  \cite{mar2}, 
 we   now  define
a rate function
$J: \ell^2 \to [0, +\infty]$ by:
for every $v\in \ell^2$,
\be\label{ratea}
J(v)
=\lim_{\delta \to 0}
\ 
\inf \left \{
I_{r,u_*}  (u):
\  r>0,  \  
u\in C([0,r], \ell^2 ),
\ u(0)=u_*, 
 \ u(r)
 \in B_{\ell^2} (v, \delta)
\right \},
\ee 
where 
$u_*$ is the unique element of $\cala$,  
$I_{r,u_*}$ is given by  \eqref{act1},
and $B_{\ell^2} (v, \delta) $ is the open ball
in $\ell^2$ centered at $v$  and of radius $\delta$.
  Given $s\ge 0$, the $s$-level set
of $J$ is the set:  
$ 
J^s
=\left \{
v \in \ell^2 : \ J(v) \le s
\right \} .
$

We first prove $J$ given by
\eqref{ratea} is a good
rate function in $\ell^2$
by the uniform tail-ends 
estimates.

   \begin{lem}\label{coj}
   If 
  \eqref{f1},
   \eqref{sig1}-\eqref{sig3} 
   and \eqref{lgc}
  are fulfilled, then
 for every $s\ge 0$,
the  level set 
$J^s$  	is compact
in $\ell^2$.
	\end{lem}

\begin{proof} 
	Let $s\ge 0$ be fixed. We  now
	prove  
	the set $J^s$ is  precompact in $\ell^2$.
	 If  $v\in J^s$, then
	 for every $n\in \N$, by \eqref{ratea} we 
	 see  that 
	there exist
	$r_n>0$  and $u_n\in
	C([0,r_n], \ell^2)$ such that
	 \be\label{coj p1}
	u_n (0) =u_{*} ,
	\quad \| u_n (r_n) -v \| <
	{\frac 1n},
\quad I_{r_n, u_{*} }
(u_n)< s+{\frac 1n}.
\ee
By \eqref{act1}  and \eqref{coj p1} we find that
  there exists
$h_n \in L^2(0, r_n; \ell^2)$ such that
 \be\label{coj p2}
 u_n(\cdot)  = u_{h_n}
(\cdot, 0, u_{*}),
\quad \text{and} \  \ 
  {\frac 12} \int_0^{r_n}
\|  h_n (s) \|^2 ds
<   
s+{\frac 1n} ,
\ee 
where $u_{h_n} (\cdot, 0, u_{*})$
is the solution
of \eqref{ctr1} corresponding to the
	  initial
	value  $u_{*} $
	and control $h_n$.
By \eqref{coj p1}-\eqref{coj p2} we get
\be\label{coj p3} 
\| u_{h_n}
(r_n, 0,  u_{*}  ) 
-v \| <{\frac 1n} \to 0.
 \ee

By   
 \eqref{coj p2} and Lemma 
 \ref{cest1}  we obtain 
\be\label{coj p4} 
\| u_{h_n} (r_n,0,   u_{*})
\|  
\le c_1,
\quad \forall \  n\in \N,
\ee
where
$c_1=c_1(\lambda,
L_\sigma, g , s)>0$
is a constant independent of $n$.
By 
\eqref{coj p3}-\eqref{coj p4} we have
 \be\label{coj p5}
 \| v\|   \le c_1,
 \quad \forall \ 
 v\in J^s.
 \ee 

By  
  \eqref{coj p2}, we get from
Lemma \ref{cest2} that
for each $\delta>0$,
 there exists
$k_0=k_0 (\delta, s)>0$
such that for all $n\in \N$,
\be\label{coj p6} 
\sum_{|i| \ge k_0}
|
(u_{h_n} (r_n, 0,  u_{*} ))_i
|^2 <
{\frac {1}{8}}\delta^2.
\ee 
By \eqref{coj p3}
and \eqref{coj p6} we get
\be\label{coj p7}
 \sum_{|i| \ge k_0}
|v_i|^2
\le {\frac {1}{8}}\delta^2,
\quad \forall \ v\in J^s.
\ee 

On the other hand, by \eqref{coj p5}
we see that
the set
$\{ (v_i)_{|i| <k_0}:  v\in J^s \}$ is 
bounded in a finite-dimensional space, and hence
precompact, which along with \eqref{coj p7}
shows that the set $J^s$ has a 
  finite open cover
of radius $ \delta$  
in $\ell^2$ for every $\delta>0$;
that is, $J^s$ is precompact in $\ell^2$.
In addition, one may    
check that the set 
  $J^s$ is  also closed in $\ell^2$,
  and thus compact,
  which completes the proof.
\end{proof}

%

  \begin{lem}\label{ldpc1}
  If \eqref{f1},
   \eqref{sig1}-\eqref{sig3} 
   and \eqref{lgc} hold,
  then
  for any  $\delta_1>0, \delta_2>0$ and $s>0$, there exists $\delta>0$
such that for all $t>0$,
$$
\left \{
u(t):
\ u\in C([0,t],\ell^2),
\ u(0) \in  \cala_\delta,
\ I_{t, u(0)}  (u)
\le s -\delta_1 
\right \}
\subseteq J^{s} _{\delta_2} , 
 $$
 where 
 $\cala_\delta$ is the $\delta$-neighborhood of
 $\cala$ and 
 $J^{s}_{\delta_2}$
 is the   $\delta_2$-neighborhood of
   $J^{s}$.
 \end{lem}

  \begin{proof} 
  	 Suppose 
  	 the statement  is false,
  	 which implies  that
    there exist
   $\delta_1>0, \delta_2>0$ and $s>0$
   such that for every $k>0$, there
   exist $t_k>0$
   and $u_k \in C([0, t_k],
   \ell^2)$
   such that
   $u_k (0)\in \cala_{\frac 1k}
   $
   and $I_{t_k, u_k(0)}
   (u_k) \le s-\delta_1$,
   but $u_k(t_k)
   \notin J^s_{\delta_2}$.
   Since $I_{t_k, u_k(0)}
   (u_k) \le s-\delta_1$,
   by \eqref{act1} we infer that
   there exists $h_k
   \in L^2(0, t_k; \ell^2)$
   such that
   \be\label{ldpc1 1}
   {\frac 12}
   \int_0^{t_{k}}
   \| h_k (t)
   \|^2 dt
   <s -{\frac 12} \delta_1,
   \quad\text{and }\
   u_k (\cdot) = u_{h_k}
   (\cdot, 0, u_k(0)),
  \ee 
   where  
   $  u_{h_k}
   (\cdot, 0, u_k(0))$
     is the solution of
     \eqref{ctr1} with
     initial value $u_k(0)$
     and control $h_k$.
     Since $u_k(t_k)
     \notin J^s_{\delta_2}$,
     by  \eqref{ldpc1 1}
     we have
      \be\label{ldpc1 2}
      u_{h_k}
     (t_k, 0, u_k(0))
     \notin J^s_{\delta_2} .
    \ee
 Due to    $u_k
 (0)\in \cala_{\frac 1k}$
 and $\cala=\{u_*\}$,
 we see that
 \be\label{ldpc1 3}
 \|  u_k(0) -u_* \|
 <{\frac 1k}.
 \ee
 
 If  $v_{h_k}(\cdot, 0, u_*)$ 
 is  the solution of 
 \eqref{ctr1}
 with initial 
 value $u_*$ 
 and control $h_k$,
 then  \eqref{cest3 2},
 \eqref{ldpc1 1} and  \eqref{ldpc1 3} 
  imply that   
   \be\label{ldpc1 4}
 \| v_{h_k} (t_k, 0, u_*) - u_{h_k}  (t_k,
 0, u_k(0) ) \|
 <{\frac 1k} e^{(\lambda - \gamma)^{-1}
 L^2_\sigma \int_0^{t_k} \| h_k (t)\|^2 dt}
  <{\frac 1k} e^{(\lambda - \gamma)^{-1}
 L^2_\sigma (2s-\delta_1) }.
  \ee
 By \eqref{ldpc1 2}
 and \eqref{ldpc1 4} we see that
 there exists $K=K(s, \delta_1, \delta_2)\in \N$ such that
 \be\label{ldpc1  5}
 v_{h_k} (t_k, 0, u_*) \notin J^s_{{\frac 12}\delta_2},
 \quad \forall \ k \ge K.
 \ee
 By \eqref{ratea} and \eqref{ldpc1 1} we have 
 for all $k\in \N$,
     \be\label{ldpc1 6}
 J(v_{h_k} (t_k,0, u_*) )
 \le I_{t_k, u_*} (v_{h_k})
 \le {\frac 12} \|h_k\|^2_{L^2(0,t_m;\ell^2)}
 <s-{\frac 12}\delta_1 .
\ee
By \eqref{ldpc1 6}
we find that
 $v_{h_k} (t_k, 0, u_* )
 \in J^s$
 for all $k\in \N$,  which is in contradiction
 with \eqref{ldpc1 5}.
\end{proof}

  \begin{lem}\label{ldpc2}
   If \eqref{f1},
   \eqref{sig1}-\eqref{sig3} 
   and \eqref{lgc} hold,
  then
   for any $R>0$ and $\delta>0$, there exists
   $t_0=t_0(R,\delta)>0$ such that
   $$
   \inf
   \left \{
   I_{t_0,u_0} (u):
   \ \|u_0\| \le R,
    u\in C([0,t_0], \ell^2),
     u(0)=u_0 , \ u(t_0)\notin
    \cala_\delta
   \right \} >0.
   $$
   \end{lem}

 \begin{proof}
 We argue by contradiction.
 If   the statement is false, then
    there exist
   $R>0$ and 
   $\delta>0$ such that for every
   $k\in \N$, there exist
   $u_{0,k} \in \ell^2$
   with $\|u_{0,k}\| \le R$
   and $u_k \in C([0, k], \ell^2)$
   such that  $u_k(0)=u_{0,k}$,
     $u(k) \notin
   \cala_\delta$
   and $I_{k, u_{0,k}}
   (u_k)<{\frac 1k}$.
   Since 
   $I_{k, u_{0,k}}
   (u_k)<{\frac 1k}$, by \eqref{act1} we see that
   there exists
   $h_k\in L^2(0,k;\ell^2)$ such that
\be\label{ldpc2 1}
   \int_0^k
   \| h_k (t)\|^2 dt
   <{\frac 4k},
   \quad
   \text{and} \ \ 
   u_k(\cdot) = u_{h_k}
   (\cdot, 0, u_{0,k}),
\ee
   where 
   $u_{h_k}
   (\cdot, 0, u_{0,k})$ 
   is  the solution 
   of \eqref{ctr1} 
   with initial condition $u_{0,k}$
   and control $h_k$.
    Since $u(k) \notin
   \cala_\delta$,
   by \eqref{ldpc2 1} we get
    \be\label{ldpc2 2}
   u_{h_k} (k, 0, u_{0,k} ) \notin \cala_\delta
   ,
  \quad \forall \ k\in \N.
   \ee

   Let $v(\cdot, 0, u_{0,k)}$ 
   be the solution of  
      \eqref{intr4}  
   with initial condition
   $u_{0,k}$.
   By choosing  
    $h_1 \equiv 0$
   and $h_2=h_k$ 
   in    \eqref{cest3 2}
   we  obtain that for all $k\in \N$,
   $$
   \| v (k, 0, u_{0,k} )
   - u_{h_k} (k, 0, u_{0,k} )\|^2
   \le 
   2(\lambda -\gamma)^{-1}
   L_\sigma^2
   \int_0^k
   \| h_k (t) \|^2 dt,
   $$
   which along with
   \eqref{ldpc2 1} shows that
   \be\label{ldpc2 3}
   \| v (k, 0, u_{0,k} )
   - u_{h_k} (k, 0, u_{0,k} )\|^2
   \le 
   8(\lambda -\gamma)^{-1}
   L_\sigma^2 k^{-1}.
 \ee
 By \eqref{ldpc2 2}-\eqref{ldpc2 3}
 we see that
 there exists $K=K(\delta)>0$ such that 
      \be\label{ldpc2 4}
  v (k, 0, u_{0,k} ) \notin \cala_{\frac 12 \delta},
  \quad \forall \ k\ge K.
   \ee

   Note that  $\cala$ is the
   global attractor of \eqref{intr4}
    and
   $\| u_{0,k}\| \le R$.
   Therefore, there exists
   $K_1=K_1(R, \delta) \ge
   K$ such that  
   $$
    v (k, 0, u_{0,k} ) \in \cala_{\frac 12 \delta},
    \quad
    \forall \  k \ge K_1,
    $$ 
   which is in contradiction with \eqref{ldpc2 4}.
     \end{proof}

 Note that bounded subsets of $\ell^2$ are not precompact
 in general. In order to establish the
 LDP of invariant measures of \eqref{intr3}
 in $\ell^2$, we 
   need to consider the weighted
     space $\ell^2_\kappa$ 
     with  weight function
      $    {\kappa}  (s) =
     \left ( 1 + s^2
     \right )^{\frac 12}$
     for   all $s\in \R$.
       The following function
      $\kappa_\delta$
      for $\delta>0$ 
      is also useful
     in the sequel: 
       $$    {\kappa_\delta}   (s) =
     \left ( 1 + \delta^2 s ^2
     \right )^{\frac 12}
     , \quad \forall
     \  s\in \R .
     $$ 
     The next lemma is concerned 
     with the exponential probability
     estimates in 
     $ \ell^2_{\kappa}
      $.

     \begin{lem}\label{ldpc3}
     	If  
     	\eqref{f1}, \eqref{sig1}-\eqref{sig4}
     	and \eqref{lgc} are fulfilled 
     	and $g\in \ell^2_\kappa$,
     	then  
     there exists $\delta\in (0,1) $ 
     such that for all $\eps\in (0,1]$
     and $u_0 \in \ell^2_{\kappa}  $,
     the solution $u^\eps (\cdot,0,  u_0)$
     of \eqref{intr3} 
     satisfies for all $t\ge 0$,
       $$
   \E \left (  
   e^{ {\frac {\delta}\eps} 
   \|{\kappa_\delta}  u^\eps (t,0, u_0 )\|^2}
   \right ) 
   \le   e^{ - t} 
   e^{ {\frac \delta\eps} 
   \| \kappa _\delta  u_0\|^2 }
   +   {\lambda} 
    e^{    (2+\eps^{-1})\lambda^{-1} }.
    $$
     \end{lem}

 \begin{proof}
    By 
    \eqref{intr3} and
       It\^{o}'s formula we  have
  $$
  d\| {\kappa_\delta}   u^\eps (t)\|^2
  + 2  
  \left ( \lambda \|    {\kappa_\delta}  u^\eps (t) \|^2
  +\sum_{j=1}^N ( B_j
  u^\eps (t),  \  B_j (\kappa_\delta ^2 u^\eps (t) ) )
     + 2    \left (  \kappa_\delta   f
  ( u^\eps (t)),   \kappa_\delta u^\eps (t) \right )
  \right ) dt
  $$
   \be\label{ldpc3 1}
  = 
    2   ( \kappa_\delta  g,  \kappa_\delta   u^\eps (t)) dt
 +  \eps 
  \|   \kappa_\delta   \sigma (u^\eps (t)) \|^2_{\call_2(\ell^2,
 \ell^2)}   dt
 + 2 \sqrt{\eps}
 ( \kappa_\delta  ^2 u^\eps (t),
 \sigma (u^\eps (t)) dW)  .
\ee
Let $\delta\in (0,1)$ be a number to be determined
and   $\xi (t,s)
   = e^{ t}
   e^{{\frac \delta{\eps}} s}$.
   Then by 
   \eqref{f1}  and
   \eqref{ldpc3 1} we  get
   $$
   d\xi (t, \|{\kappa_\delta} u^\eps (t)\|^2)
     \le
   \xi (t, \|{\kappa_\delta} u^\eps (t)\|^2)
   \left (
   1 +{\frac {2\delta^2}\eps}\|{\kappa_\delta} 
   \sigma(u^\eps (t))\|^2_{\call_2
   	(\ell^2,\ell^2)} \|{\kappa_\delta} u^\eps (t) \|^2
   	\right )
   	$$
   	$$
   +
    \xi (t, \|{\kappa_\delta} u^\eps (t)\|^2)
   \left (\delta \|{\kappa_\delta} 
   \sigma(u^\eps (t)) \|^2_{\call_2(\ell^2,\ell^2)}
   +{\frac {2\delta}\eps}
   (\kappa_\delta g,  \kappa_\delta u^\eps (t) )
   \right )dt
   $$
   $$
   -{\frac {2\delta}\eps}
  \xi    (t, \|{\kappa_\delta}
  u^\eps (t)\|^2)
   \left (
   \lambda \|{\kappa_\delta} u^\eps (t) \|^2
   +\sum_{j=1}^N ( B_j
  u^\eps (t),  \  B_j (\kappa_\delta ^2 u^\eps (t) ) )
   \right )  dt
   $$
   \be\label{ldpc3 2}
    +{\frac {2\delta}
   	{\sqrt{\eps}}}  \xi (t, \|
   {\kappa_\delta}  u^\eps (t)\|^2)
   ({\kappa_\delta} u^\eps (t),
   {\kappa_\delta}  \sigma(u^\eps (t)) dW).
   \ee
   Next, we estimate the terms in \eqref{ldpc3 2}.
 Note that
 $$-
 \sum_{j=1}^N ( B_j
  u^\eps (t),  \  B_j (\kappa_\delta ^2 u^\eps (t) ) )
 =-\sum_{j=1}^N
 \sum_{i\in \Z}
 (B_j
  u^\eps (t))_i \ 
  (B_j (\kappa_\delta ^2 u^\eps (t) ))_i
  $$
  \begin{align*}
  = &-\sum_{j=1}^N
 \sum_{i\in \Z}  
 \left (
 (u^\eps (t))_{(i_1,\ldots, i_j +1, \ldots, i_N)}
 -
 (u^\eps (t))_{(i_1,\ldots, i_j , \ldots, i_N)}
 \right ) \\
 & \times
 \left (
 (\kappa_\delta ^2u^\eps (t))_{(i_1,\ldots, i_j +1, \ldots, i_N)}
 -
 (\kappa_\delta ^2 u^\eps (t))_{(i_1,\ldots, i_j , \ldots, i_N)}
 \right ) \\ 
  = & 
  -\sum_{j=1}^N
 \sum_{i\in \Z}  
 \left (
 (u^\eps (t))_{(i_1,\ldots, i_j +1, \ldots, i_N)}
 -
 (u^\eps (t))_{(i_1,\ldots, i_j , \ldots, i_N)}
 \right ) \\
 & \times
 \left (
 (1+\delta^2 (|i|^2 +2i_j +1) )
 (u^\eps (t))_{(i_1,\ldots, i_j +1, \ldots, i_N)}
 -
  (1+\delta^2 |i|^2   )
  ( u^\eps (t))_{(i_1,\ldots, i_j , \ldots, i_N)}
 \right ) \\
 = & 
  - \delta^2 \sum_{j=1}^N
 \sum_{i\in \Z} 
 (2i_j +1)   
 (u^\eps (t))_{(i_1,\ldots, i_j +1, \ldots, i_N)}
   (
  (u^\eps (t))_{(i_1,\ldots, i_j +1, \ldots, i_N)}
 -
 (u^\eps (t))_{(i_1,\ldots, i_j , \ldots, i_N)}
 ) 
   \end{align*}
\be\label{ldpc3 3}
 -
 \sum_{j=1}^N
 \sum_{i\in \Z} 
  (1+\delta^2  |i|^2  )
 ( (u^\eps (t))_{(i_1,\ldots, i_j +1, \ldots, i_N)}
 - 
  ( u^\eps (t))_{(i_1,\ldots, i_j , \ldots, i_N)}
 )^2.
 \ee
 
   For the first term
   on the right-hand side of \eqref{ldpc3 3}   we have
      $$  - \delta^2 \sum_{j=1}^N
 \sum_{i\in \Z} 
 (2i_j +1)   
 (u^\eps (t))_{(i_1,\ldots, i_j +1, \ldots, i_N)}
   (
  (u^\eps (t))_{(i_1,\ldots, i_j +1, \ldots, i_N)}
 -
 (u^\eps (t))_{(i_1,\ldots, i_j , \ldots, i_N)}
 ) 
 $$
 $$
 \le
  \delta^2 \sum_{j=1}^N
 \sum_{i\in \Z} 
 2 |i_j |   \left |
 (u^\eps (t))_{(i_1,\ldots, i_j +1, \ldots, i_N)}
   (
  (u^\eps (t))_{(i_1,\ldots, i_j +1, \ldots, i_N)}
 -
 (u^\eps (t))_{(i_1,\ldots, i_j , \ldots, i_N)}
 )  \right |
 $$
 $$
 +
   \delta^2 \sum_{j=1}^N
 \sum_{i\in \Z} 
 \left |
 (u^\eps (t))_{(i_1,\ldots, i_j +1, \ldots, i_N)}
   (
  (u^\eps (t))_{(i_1,\ldots, i_j +1, \ldots, i_N)}
 -
 (u^\eps (t))_{(i_1,\ldots, i_j , \ldots, i_N)}
 ) 
 \right |
 $$
 $$
 \le
  \delta^2 \sum_{j=1}^N
 \sum_{i\in \Z} 
   |i_j  |^2  
   (
  (u^\eps (t))_{(i_1,\ldots, i_j +1, \ldots, i_N)}
 -
 (u^\eps (t))_{(i_1,\ldots, i_j , \ldots, i_N)}
 )^2 
 $$
 $$
 +
   \delta^2 \sum_{j=1}^N
 \sum_{i\in \Z} 
 \left ({\frac 52}
  |
 ( u^\eps (t))_{(i_1,\ldots, i_j +1, \ldots, i_N)}|^2
  +  {\frac 12}  |
 (u^\eps (t))_{(i_1,\ldots, i_j , \ldots, i_N)}
   |^2 
 \right  )
 $$
 $$
 \le
  \delta^2 \sum_{j=1}^N
 \sum_{i\in \Z} 
   |i  |^2  
   (
  (u^\eps (t))_{(i_1,\ldots, i_j +1, \ldots, i_N)}
 -
 (u^\eps (t))_{(i_1,\ldots, i_j , \ldots, i_N)}
 )^2  
 +3N
   \delta^2  \| u^\eps (t)\|^2
 $$
 $$
 \le
  \delta^2 \sum_{j=1}^N
 \sum_{i\in \Z} 
   |i  |^2  
   (
  (u^\eps (t))_{(i_1,\ldots, i_j +1, \ldots, i_N)}
 -
 (u^\eps (t))_{(i_1,\ldots, i_j , \ldots, i_N)}
 )^2  
 +3N
   \delta^2  \| \kappa_\delta u^\eps (t)\|^2,
 $$
 which along with
  \eqref{ldpc3 3} shows that
 \be\label{ldpc3 4}
 -
 \sum_{j=1}^N ( B_j
  u^\eps (t),  \  B_j (\kappa_\delta ^2 u^\eps (t) ) )
  \le
  3N
   \delta^2  \| \kappa_\delta u^\eps (t)\|^2.
 \ee

 On the other hand, by \eqref{sig4} we have
 for all $\delta \in (0,1)$,
  \be\label{ldpc3 5}
 \|{\kappa_\delta} 
   \sigma(u^\eps (t))\|^2_{\call_2
   	(\ell^2,\ell^2)}
   	\le 
   	\| (1+ |\cdot|^2)^{\frac 12}
   \sigma(u^\eps (t))\|^2_{\call_2
   	(\ell^2,\ell^2)}
   	\le L_\sigma^2.
\ee
By Young's inequality we have
for all $\delta\in (0,1)$,
 \be\label{ldpc3 6}
 2 (\kappa_\delta g,
  \kappa_\delta u^\eps (t) )
  \le
  {\frac 12} \lambda \| \kappa_\delta u^\eps (t)\|^2
  + 2\lambda^{-1} 
 \| \kappa_\delta g\|^2
  \le
  {\frac 12} \lambda \| \kappa_\delta u^\eps (t)\|^2
  + 2\lambda^{-1} 
 \| g\|_\kappa^2.
 \ee

 It follows from \eqref{ldpc3 2},
 \eqref{ldpc3 4}-\eqref{ldpc3 6}
 that
  $$
   d\xi (t, \|{\kappa_\delta} u^\eps (t)\|^2)
     \le
   \xi (t, \|{\kappa_\delta} u^\eps (t)\|^2)
   \left (
   1 + 
   \delta L_\sigma^2 
   +{\frac {2\delta}{\eps \lambda}}
    \|g\|^2_\kappa
   	\right )dt
   	$$
    $$
   -{\frac {\delta}\eps}
  \xi    (t, \|{\kappa_\delta}
  u^\eps (t)\|^2)
     \left (2
   \lambda 
   -2\delta L_\sigma^2
   -{\frac 12}\lambda 
   - 6N \delta^2 \right )
   \|{\kappa_\delta} u^\eps (t) \|^2
     dt
   $$
   \be\label{ldpc3 7}
    +{\frac {2\delta}
   	{\sqrt{\eps}}}  \xi (t, \|
   {\kappa_\delta}  u^\eps (t)\|^2)
   ({\kappa_\delta} u^\eps (t),
   {\kappa_\delta}  \sigma(u^\eps (t)) dW).
   \ee
Fix a positive number   $\delta$ such that
 $$
 0<\delta <\min \{1, {\frac  {\lambda}{1+ 2\|g\|^2_\kappa}} ,
 L^{-2}_\sigma
 \}
 \quad \text{and} \ \ 
 (L^2_\sigma +3N \delta) \delta
 <{\frac 14} \lambda.
 $$
 Then by \eqref{ldpc3 7} we get
  $$
   d\xi (t, \|{\kappa_\delta} u^\eps (t)\|^2)
     \le
   \lambda \xi (t, \|{\kappa_\delta} u^\eps (t)\|^2)
   \left ((2+\eps^{-1})\lambda^{-1}  -{\frac \delta{\eps}}  
   \|{\kappa_\delta} u^\eps (t) \|^2
   \right )
     dt
   $$
   \be\label{ldpc3 8}
    +{\frac {2\delta}
   	{\sqrt{\eps}}}  \xi (t, \|
   {\kappa_\delta}  u^\eps (t)\|^2)
   ({\kappa_\delta} u^\eps (t),
   {\kappa_\delta}  \sigma(u^\eps (t)) dW).
   \ee
 Using a stopping time if necessary,
 by   \eqref{ldpc3 8}   we get
   $$
   \E \left ( e^{  t}
   e^{ {\frac \delta\eps} 
   \|{\kappa_\delta}  u^\eps (t)\|^2}
   \right )
   $$
   \be\label{ldpc3 9}
   \le  
   \E \left (
   e^{ {\frac \delta\eps} 
   \|{\kappa_\delta}  u^\eps (0)\|^2}
   \right )
   +   {\lambda} 
   \E \left (\int_0^t
    e^{ 
    	 s}
   e^{ {\frac \delta\eps} 
   \|{\kappa_\delta}  u^\eps (s)\|^2}
   \left ( (2+\eps^{-1})\lambda^{-1}
   - 
    {\frac {\delta}\eps}  
     \|{\kappa_\delta} u^\eps (s) \|^2 \right ) ds
     \right ).
     \ee
     Note that $e^r (a-r) \le e^{a-1}
     \le e^a$ for all $a, r\in \R$.
     We  obtain from \eqref{ldpc3 9}
     that
      $$
   \E \left ( e^{  t}
   e^{ {\frac \delta\eps} 
   \|{\kappa_\delta}  u^\eps (t)\|^2}
   \right ) 
   \le  
   \E \left (
   e^{ {\frac \delta\eps} 
   \|{\kappa_\delta}  u^\eps (0)\|^2}
   \right )
   +   {\lambda} 
    e^{ t+  (2+\eps^{-1})\lambda^{-1} },
    $$
    which shows that for all $t\ge 0$,
     $$
   \E \left (  
   e^{ {\frac \delta\eps} 
   \|{\kappa_\delta}  u^\eps (t)\|^2}
   \right ) 
   \le   e^{ - t} 
   e^{ {\frac \delta\eps} 
   \|{\kappa_\delta}  u^\eps (0)\|^2}
   +   {\lambda} 
    e^{    (2+\eps^{-1})\lambda^{-1} }.
    $$
    This completes the proof. 
   \end{proof}

   If \eqref{lgc} holds,
   then for any $\eps <(\lambda -\gamma)
   L^{-2}_\sigma$, by \eqref{f1}
   and \eqref{sig2} we  obtain from
   \eqref{intr3} that for all 
   $u_{0,1}, u_{0,2}
   \in \ell^2$  and $t\ge 0$,
   $$
   \E \left (
   \| u^\eps (t,0, u_{0,1})
   - u^\eps (t, 0, u_{0,2})
   \|^2
   \right )
   \le e^{-(\lambda -\gamma) t}
   \| u_{0,1} -u_{0,2} \|^2,
   $$
  and hence 
   the invariant measure 
   $\mu^\eps$
   of \eqref{intr3}
   is  unique in this case,
   which  
   is  the weak limit of
   the sequence
   ${\frac 1k}
   \int_0^k
  \call (u^\eps (t,0, 0)) dt$
  as $k\to \infty$,
  where $u^\eps (\cdot, 0,0)$ is the
  solution of \eqref{intr3}
  with initial value $0$ at initial time $0$,
  and $\call (u^\eps (t,0, 0))$ is the distribution
  of  $ u^\eps (t,0, 0)$.
  Then by 
     Lemma 
  \ref{ldpc3} with $u_0=0$ we infer that 
  the invariant measure 
  of \eqref{intr3}
  is  supported on $\ell^2_\kappa$.

  \begin{lem}\label{ldpc4}
  If \eqref{f1},
  \eqref{sig1}-\eqref{sig4}
  and \eqref{lgc} are fulfilled
  and $g\in \ell^2_\kappa$,
  then    
  	$$
  	\lim_{R\to \infty}
  	\limsup_{\eps \to 0}
  	\eps  \ln \mu^\eps (\ell^2 \setminus {\overline{B}}
  	_{\ell^2_\kappa}
  	(0, R) )  
  	=-\infty,
  	$$
  	where
  	${\overline{B}}
  	_{\ell^2_\kappa}
  	(0, R) ) $
  	is the closed ball in $ \ell^2_\kappa$
  	centered at $0$ and of radius $R$.
  \end{lem}

    \begin{proof}
      Let $\delta\in (0,1)$ be 
      the small number as
      given by Lemma \ref{ldpc3}.
       Given  $u_0 \in  \ell^2 $ 
       and $t>0$, we  have
     $$
     P \left (
     u^\eps (t,0,
     u_0) \notin    {\overline{B}}
  	_{\ell^2_\kappa}
  	(0, R ) 
      \right ) 
      =
       P \left (
     \|u^\eps (t,0, u_0)  \|^2_{\ell^2_{\kappa}
     } >  R^2
     \right ) 
     $$
     $$
      \le
      P \left (
     \|\kappa_\delta  u^\eps (t, 0, u_0)  \|^2 
     >  R^2\delta^2
     \right ) 
       =
      P \left (
    e^{ 
     {\frac \delta\eps} \| \kappa_\delta  u^\eps (t,0,
      u_0 ) \|^2  }
       >
     e^{
     {\frac  {R ^2  \delta^3}\eps} 
     }
     \right  )
       $$
   \be\label{ldpc4 1}
       \le
       e^{-
     {\frac  {R ^2  \delta^3}\eps} 
     }
     \E \left (
      e^{ 
     {\frac \delta\eps} \|\kappa_\delta 
      u^\eps (t, 0, u_0 ) \|^2 }
     \right ).
    \ee
      By Lemma \ref{ldpc3} and \eqref{ldpc4 1}
      we get 
           \be\label{ldpc4 2}
     P \left (
     u^\eps (t,0,
     u_0) \notin    {\overline{B}}
  	_{\ell^2_\kappa}
  	(0, R ) 
      \right ) 
     \le
        e^{ - t}   e^{-
     {\frac  {R ^2  \delta^3}\eps} 
     } 
   e^{ {\frac \delta\eps} 
   \| \kappa _\delta  u_0\|^2 }
   +   {\lambda} 
    e^{    (2+\eps^{-1})\lambda^{-1} }   e^{-
     {\frac  {R ^2  \delta^3}\eps} 
     }.
\ee
      It follows from \eqref{ldpc4 2} and 
        the invariance of $\mu^\eps$  
        that 
     for all $t\ge 0$,
     $$
     \mu^\eps ( (\ell^2 \setminus {\overline{B}}
  	_{\ell^2_\kappa}
  	(0, R) )    )
     =
     \int_{\ell^2}
       P \left (
     u^\eps (t,0,
     u_0) \notin    {\overline{B}}
  	_{\ell^2_\kappa}
  	(0, R ) 
      \right ) 
      d\mu^\eps (u_0)
     $$
               \be\label{ldpc4 3}
      =
     \int_{\ell^2_\kappa}
       P \left (
     u^\eps (t,0,
     u_0) \notin    {\overline{B}}
  	_{\ell^2_\kappa}
  	(0, R ) 
      \right ) 
      d\mu^\eps (u_0),
    \ee
    The last equality follows from the fact
    that
    $\mu^\eps$ is supported on $\ell^2_\kappa$.
      By  \eqref{ldpc4 2}-\eqref{ldpc4 3}
      and   the Fatou theorem we get
       $$ 
     \mu^\eps ( (\ell^2 \setminus {\overline{B}}
  	_{\ell^2_\kappa}
  	(0, R) )    )
  	\le
  	\int_{\ell^2_\kappa}
       \limsup_{t \to \infty}
       P \left (
     u^\eps (t,0,
     u_0) \notin    {\overline{B}}
  	_{\ell^2_\kappa}
  	(0, R ) 
      \right ) 
      d\mu^\eps (u_0) 
      $$
        \be\label{ldpc4 4}
    \le  
      {\lambda} 
    e^{    (2+\eps^{-1})\lambda^{-1} }   e^{-
     {\frac  {R ^2  \delta^3}\eps} 
     }.
     \ee 
     Then the desired limit follows 
     from \eqref{ldpc4 4} immediately.
     \end{proof}

    In the next two sections,
    we prove the LDP   
      of
      the family $\{\mu^\eps\}$
      of  invariant measures
      of \eqref{intr3} as $\eps \to 0$,
      based on the results obtained
      in the previous sections.

\section{LDP lower bound of invariant measures} 
\setcounter{equation}{0}

This section is devoted to the
      LDP lower bound
      of  the family  $\{\mu^\eps\}$
      of
  invariant measures of \eqref{intr3} as
  $\eps\to 0$. The argument is standard
  (see, e.g., \cite{mar2, sow1}).
  We here sketch the idea of the proof 
  only for the sake
  of reader's convenience.

  We will show that 
  for any  $z\in \ell^2$,
  $s_1>0$
  and $s_2>0$, there exists
  $\eps_0$ such that
  \be\label{ldplb a1}
  \mu^\eps
  (B_{\ell^2}  (z, s_1))
  \ge e^{
  - {\frac {J (z) +s_2}
  {\eps}
  }
  },
  \quad \forall \    \eps \le \eps_0,
  \ee
  where $B_{\ell^2} (z,s_1)$
  is the open ball in $\ell^2$
  centered at $z$ and of radius $s_1$.
  Without loss of generality, we may assume that
   $J(z) <\infty$. In this case,  by \eqref{ratea}
   we see that
   there exist $r_0>0$
   and $\widetilde{u}
   \in C([0, r_0], \ell^2)$
   such that
   $\widetilde{u} (0) =u_*$,
    $\widetilde{u} (r_0)
   \in B_{\ell^2} (z, {\frac 14} s_1)$,
   and $I_{r_0, u_*} (\widetilde{u})
   < J(z) +
     {\frac 16} s_2.
     $ 
     By \eqref{act1} we infer that
     there exists
   $h\in L^2(0,r_0; \ell^2)$  such that
     \be\label{ldplb a2}
     {\frac 12}
     \int_0^{r_0} \| h(t) \|^2
     dt < 
     I_{r_0, u_*} (\widetilde{u})
      +
     {\frac 16} s_2,
     \quad  
     \text{and}
     \quad 
       \widetilde{u} (\cdot)
       =u_h(\cdot, 0, u_*),
     \ee
    where $u_h(\cdot, 0, u_*)$
    is the solution of \eqref{ctr1}
     with $u_0=u_*$.
     By  \eqref{ldplb a2}
     we get
      \be\label{ldplb a3}
     {\frac 12}
     \int_0^{r_0} \| h(t) \|^2
     dt <  J(z)
      +
     {\frac 13} s_2,
     \quad  
     \text{and}
     \quad 
      \| u_h(r_0, 0, u_*)-  z \|
     <{\frac 14} s_1.
     \ee
    By \eqref{cest3 2} 
    and \eqref{ldplb a3} we find   
    that for any $y\in \ell^2$,
    $$
     \| u_h(r_0,  0, y) -u_h(r_0, 0, u_*)\|^2
  \le
  e^{2(\lambda -\gamma)^{-1}
  L^2_\sigma \int_0^{r_0}
  \|h(t)\|^2 dt }\| y-u_*\|^2
   $$
   $$
   \le
  e^{4(\lambda -\gamma)^{-1}
  L^2_\sigma   (J(z) +{\frac 13 s_2}) }\|y-u_*\|^2,
   $$
  and hence there exists
  $\eta=\eta(z, s_1, s_2)>0$ such that
  for all $\|y-u_*\|<\eta$, 
 \be\label{ldplb a4}
     \| u_h(r_0,  0, y) -u_h(r_0, 0, u_*)\|
<
  {\frac 14} s_1.
 \ee
 By \eqref{ldplb a3}-\eqref{ldplb a4}
 we have 
$$
    \| u_h (r_0, 0, y  ) -z\|<{\frac 12}
   s_1,
    \quad \text{for all }
     y\in  \ell^2
    \ \text{with}
    \ 
    \| y-u_*\|
    <\eta,
    $$
    which implies that if 
   $\|y-u_*\|<\eta$, then
   $$
   \| u^\eps (r_0, 0, y   )- z\|
   < \|u^\eps (r_0, 0, y  )
   -u_h(r_0,0, y )\|
    +{\frac 12} s_1,
   $$  
   and hence
  \be\label{ldplb a5}
   P
    \left (
     \| u^\eps (r_0,0, y)- u_h(r_0,0,y) \|<
     {\frac 12}
     s_1
    \right )
    \le
    P
    \left (
     \| u^\eps (r_0,0, y)- z\|<s_1
    \right ).
    \ee
    By
        \eqref{uldp 1} and
     \eqref{ldplb a3}
     we infer that  
    \be\label{ldplb a6}
     P
    \left ( 
     \| u^\eps (\cdot, 0, y )- u_h(\cdot,0, y  ) \|
     _{C([0,r_0], \ell^2)}<
     {\frac 12}
     s_1
    \right )
    \ge e^{
    -{\frac {I_{r_0, y} (u_h(\cdot, 0,y  )) +{\frac 13} s_2}
    {\eps}
    }
    }
      \ge e^{
    -{\frac {
    J(z)  +{\frac 23} s_2}
    {\eps}
    }
    }.
    \ee
By \eqref{ldplb a5}-\eqref{ldplb a6}
 we  see that if $\| y-u_*\|<
 \eta  $, then
   \be\label{ldplb a7}
    P
    \left (
     \| u^\eps (r_0, 0, y) -z\|<s_1 
    \right )
    \ge 
    e^{\frac {s_2}{3\eps}}
       e^{
    -{\frac {
    J(z)  +  s_2}
    {\eps}
    }
    }.
    \ee
   By  \eqref{ldplb a7}
   and the invariance of 
     $\mu^\eps$  we get
     $$
    \mu^\eps
    (B_{\ell^2} (z, s_1))
    =
    \int_{\ell^2}
    P (u^\eps (r_0,0, y) \in 
    B_{\ell^2} (z, s_1)) \mu^\eps
    (d y)
    $$
\be\label{ldplb a8}
    \ge
     \int_{B_{\ell^2}
     (u_*, \eta)  }
    P (u^\eps (r_0,0, y) \in 
    B_{\ell^2} (z, s_1)) \mu^\eps
    (d y)
    \ge
    \mu^\eps
    (B_{\ell^2}(u_*, 
     \eta ) )
     e^{\frac {s_2}{3\eps}}
       e^{
    -{\frac {
    J(z)  +  s_2}
    {\eps}
    }
    }.
   \ee
    Since $\cala =\{u_*\}$, by
   Theorem \ref{limi} we have
     $\mu^\eps \to \delta_{u_*}$ weakly,  
     and thus, 
$$
\liminf_{\eps
\to 0}    \left (   e^{\frac {s_2}{3\eps}}
  \mu^\eps
    (B_{\ell^2}(u_*, 
     \eta ) )
     \right )
     \ge
     \lim _{\eps
\to 0}      e^{\frac {s_2}{3\eps}}
  = \infty,  
    $$ 
    which along with \eqref{ldplb a8}
    yields \eqref{ldplb a1}, and thus completes
    the proof of the LDP lower bound of $\{\mu^\eps\}$.

\section{LDP upper bound
of invariant measures} 
\setcounter{equation}{0}

This section is devoted to the 
LDP upper bound   
of the family $\{\mu^\eps \}$
  of invariant measures of
\eqref{intr3} as $\eps \to 0$.
Again,
 the argument follows from 
  \cite{mar2, sow1}, and we just
   sketch the proof 
   for  reader's convenience.

We will show that  
 for  any    $s_1>0 $,
$ s_2 >0$  and
$s>0$,  
there exists
$\eps_0>0$ such that
 \be\label{ldpub a1}
\mu^\eps (\ell^2 \setminus
J^s_{s_1})
\le e^{
-{\frac {s-s_2}
{\eps}
}
},
\quad \forall \ \eps  \le \eps_0,
\ee
where $J^s_{s_1}$
is the   $s_1$-neighborhood
of $J^s$.

By Lemma \ref{ldpc1},  there exists
$\eta>0$ such that
for all $t>0$,
\be\label{ldpub a2}
\left \{
u(t):\ 
u\in C([0,t], \ell^2), 
\ u(0)\in  \cala_\eta,
\ I_{t, u(0)} (u)
\le s-  {\frac 14} s_2
\right \}
\subseteq J^s_{  {\frac 12} s_1 }.
\ee
 By Lemma \ref{ldpc4}, 
  there
exist
$R=R(s)  >0$
and $\eps_1=\eps_1 (s) >0$
such that for all
$\eps  \le \eps_1$,
\be\label{ldpub a3}
\mu^\eps 
\left (\ell^2 \setminus  {\overline{B}}_{\ell^2_\kappa}
(0, R)  \right )
\le e^{-{\frac s \eps}}.
\ee  
By Lemma \ref{ldpc2}, 
 we infer that there exists $t_0=t_0 (R, \eta)>0$ such that
\be\label{ldpub a4}
\delta
= \inf \left \{
I_{t_0, u_0} (u):\
 \| u_0 \| \le R,   
\ u\in C([0,t_0], \ell^2),
\   u(0) =u_0, 
\ u(t_0) \notin \cala_\eta
\right \}>0.
\ee
For every  $n\in \N$,
denote by 
\be\label{ldpub a5}
\caly_n
=\left \{
u\in C([0, n t_0], \ell^2):
\ \|u(0)\| \le R  ,
\  u(kt_0)\notin  \cala_\eta, \ 
 \|u(kt_0) \| \le R, \ \forall \ k=1,\ldots, n
\right \}.
\ee  
By \eqref{ldpub a4}
and \eqref{ldpub a5}
one can verify that   
if $n> {\frac s{\delta}}  $, then
\be\label{ldpub a6}
\inf
\left \{
I_{nt_0, u_0} (u): \   \| u_0\| \le R,
   \  u\in \caly_n, \  u(0)=u_0
\right \} >s .
\ee

  Let   
$n_0  $ be a fixed positive integer  such that
$n_0 > {\frac s\delta}$, which along with 
  \eqref{ldpub a6}
implies that
\be\label{ldpub a10} 
\inf
\left \{
I_{n_0 t_0, u_0} (u): \   \| u_0\| \le R,
   \  u\in \caly_{n_0}, \ u(0)=u_0
\right \} > s.
\ee
Note that if $ \| u_0\|_{\ell^2_\kappa}
\le R$, then $\|u_0\| \le R$, and hence
by \eqref{ldpub a10} we have
  \be\label{ldpub a10a} 
\inf
\left \{
I_{n_0 t_0, u_0} (u): \   \| u_0\|_{\ell^2_\kappa} \le R,
   \  u\in \caly_{n_0}, \ u(0)=u_0
\right \} > s.
\ee
It is easy to verify that
${\overline{B}}_{\ell^2_\kappa} (0, R)$
is  a compact
subset of $\ell^2$. Then
applying  Proposition \ref{uldp_DZ}
to  the compact set ${\overline{B}}_{\ell^2_\kappa} (0, R)$
and the closed subset
  $\caly_{n_0}$  of 
 $C([0, n_0 t_0], \ell^2)$, 
by \eqref{ldpub a10a} 
     we 
infer that there exists $\eps_2
\in (0, \eps_1)$ such that
for all $\eps \le \eps_2$,
\be\label{ldpub a11}
\sup_{\| u_0 \|_{\ell^2_\kappa} \le R  }
P
\left (
u^\eps (\cdot, 0, u_0   )\in \caly_{n_0}
\right )
\le e^{-
{\frac {s}
{\eps}
 }}.
\ee

For convenience, we set
 $r_0
=(1+ n_0)t_0$. 
Since $\mu^\eps$
is  an invariant measure, we have
$$
\mu^\eps
(\ell^2 \setminus
J^s_{s_1})
=\int_{\ell^2 }
P (u^\eps (r_0, 0,  u_0 )
\notin J^s_{s_1} )
\mu^\eps (d u_0 )
$$
$$
=\int_{  \| u_0 \|_{\ell^2_\kappa }
>R}
P (u^\eps (r_0, 0, u_0   )
\notin J^s_{s_1} )
\mu^\eps (d u_0 )
+
 \int_{  \| u_0 \|_{\ell^2_\kappa }
\le R}
P (u^\eps (r_0,  0, u_0  )
\notin J^s_{s_1} )
\mu^\eps (d u_0 )
$$ 
$$
=\int _{  \| u_0 \|_{\ell^2_\kappa }
>R}
P (u^\eps (r_0, 0, u_0 )
\notin J^s_{s_1} )
\mu^\eps (d u_0 )
$$
$$
+
 \int_{  \| u_0 \|_{\ell^2_\kappa }
\le R}
P (u^\eps (r_0, 0, u_0 )
\notin J^s_{s_1},
\ \ u^\eps (\cdot,0, u_0   ) \in \caly_{n_0}
 )
\mu^\eps (d u_0 )
$$ 
 $$
+
 \int_{  \| u_0 \|_{\ell^2_\kappa }
\le R}
P (u^\eps (r_0, 0, u_0 )
\notin J^s_{s_1},
\ \ u^\eps (\cdot, 0, u_0 ) \notin \caly_{n_0}
 )
\mu^\eps (d u_0)
$$
$$
\le
\mu^\eps (  \ell^2\setminus  {\overline{B}}_{\ell^2_\kappa}
(0, R) )  
+ 
 \int_{ \| u_0 \|_{\ell^2_\kappa }
\le R}
P (  \ u^\eps (\cdot, 0, u_0 ) \in \caly_{n_0}
 )
\mu^\eps (d u_0 )
$$
$$
+
 \int_{ \| u_0 \|_{\ell^2_\kappa }
\le R }
P (u^\eps (r_0, 0, u_0 )
\notin J^s_{s_1},
\ \ u^\eps (\cdot, 0, u_0 ) \notin \caly_{n_0}
 )
\mu^\eps (d u_0 )
$$
$$
\le 
\mu^\eps (  \ell^2\setminus  {\overline{B}}_{\ell^2_\kappa}
(0, R) )  
+ 
  \sup_{ \| u_0 \|_{\ell^2_\kappa }
\le R  }
P (  u^\eps (\cdot, 0, u_0 ) \in \caly_{n_0}
 ) 
$$
\be\label{ldpub a11a}
+
 \int_{  \| u_0 \|_{\ell^2_\kappa }
\le R  }
P (u^\eps (r_0, 0, u_0 )
\notin J^s_{s_1},
\ \ u^\eps (\cdot, 0, u_0 ) \notin \caly_{n_0}
 )
\mu^\eps (d u_0 ).
\ee
By  
\eqref{ldpub a3}
\eqref{ldpub a11}
and \eqref{ldpub a11a}   we get 
\be\label{ldpub a15}
\mu^\eps
(\ell^2 \setminus
J^s_{s_1})
\le
 2 e^{- {\frac {s}{\eps}}}
 +
 \int_{ \| u_0 \|_{\ell^2_\kappa }
\le R}
P (u^\eps (r_0, 0, u_0 )
\notin J^s_{s_1},
  \ u^\eps (\cdot, 0, u_0 ) \notin \caly_{n_0}
 )
\mu^\eps (d u_0).
 \ee

By 
    \eqref{ldpub a5}  we have 
 $$
 \int_{\|u_0 \| \le R}
P\left 
 (u^\eps (r_0, 0, u_0 )
\notin J^s_{s_1},
  \ u^\eps (\cdot, 0, u_0 ) \notin \caly_{n_0}
 \right ) \ 
\mu^\eps (d u_0 )
 $$
 $$
 \le
 \sum_{k=1}^{n_0}
  \int_{\|u_0 \| \le R}
P \left (u^\eps (r_0, 0, u_0 )
\notin J^s_{s_1},
\  u^\eps (kt_0, 0, u_0 ) \in  \cala_{\eta}
 \right )\ 
\mu^\eps (d u_0 )
 $$
 \be\label{ldpub a21}
 +
 \sum_{k=1}^{n_0}
  \int_{\|u_0 \| \le R}
P \left (u^\eps (r_0, 0, u_0 )
\notin J^s_{s_1},
\  \|u^\eps (kt_0, 0, u_0 ) \|>R
 \right )\ 
\mu^\eps (d  u_0 ).
\ee 
 By  the invariance of $\mu^\eps$,   we 
 have
$$
 \sum_{k=1}^{n_0}
  \int_{\|u_0 \| \le R}
P \left (u^\eps (r_0, 0, u_0 )
\notin J^s_{s_1},
\  \|u^\eps (kt_0, 0, u_0 ) \|>R
 \right )\ 
\mu^\eps (d u_0 )
$$
$$
\le
 \sum_{k=1}^{n_0}
  \int_{\|u_0 \| \le R}
P \left (   \|u^\eps (kt_0, ,0, u_0 ) \|>R
 \right )\ 
\mu^\eps (d u_0 )
\le
 \sum_{k=1}^{n_0}
  \int_{\ell^2 }
P \left (   \|u^\eps (kt_0, 0, u_0 ) \|>R
 \right )\ 
\mu^\eps (d u_0 )
$$
$$
=n_0 \mu^\eps 
\left (  \ell^2\setminus {\overline{B}}_{\ell^2}  (0,R)
\right )
\le n_0 \mu^\eps 
(\ell^2 \setminus  {\overline{B}}
_{\ell^2_\kappa} (0, R) ),
  $$
which together with
  \eqref{ldpub a3} shows that
\be\label{ldpub a22}
 \sum_{k=1}^{n_0}
  \int_{\|u_0 \| \le R}
P \left (u^\eps (r_0, 0, u_0 )
\notin J^s_{s_1},
\  \|u^\eps (kt_0, 0, u_0 ) \|>R
 \right )\ 
\mu^\eps (d u_0 )
\le n_0 e^{
-{\frac {s}{\eps}}}.
\ee 

We now write 
  the transition probability 
of $u^\eps (t, 0, u_0)$ as
  $P(0, u_0; t, \cdot)$.
By   the Markov property, 
we find that
  the first
 term on the right-hand side
of \eqref{ldpub a21} satisfies
$$ 
 \sum_{k=1}^{n_0}
  \int_{\|u_0 \| \le R}
P \left (u^\eps (r_0, 0, u_0 )
\notin J^s_{s_1},
\  u^\eps (kt_0, 0, u_0 ) \in  \cala_{\eta}
 \right )\ 
\mu^\eps (d u_0 )
$$
$$
=
\sum_{k=1}^{n_0}
  \int_{\|u_0 \| \le R}
\left (
\int_{y\in \cala_\eta}
P
\left (
u^\eps (r_0 -kt_0, y
) \notin J^s_{s_1}
\right )
p(0,u_0; kt_0, dy)
\right ) \ \mu^\eps (du_0 )
$$
 \be\label{ldpub a23}
\le  \sum_{k=1}^{n_0}
\left ( 
\sup_{y\in \cala_\eta}
P
\left (
u^\eps (r_0 -kt_0,  y
) \notin J^s_{s_1}
\right ) 
\right ) .
\ee

 By \eqref{ldpub a2} we see that
 \be\label{ldpub a23a}
\left \{
u( r_0 -kt_0):\ 
u\in C([0, r_0 -kt_0], \ell^2), 
\ u(0)\in  \cala_\eta,
\ I_{r_0 -kt_0, u(0)  }(u)
\le s-  {\frac 14} s_2
\right \}
\subseteq J^s_{  {\frac 12} s_1 }.
\ee
By \eqref{ldpub a23a} we find that
for any  $y\in \cala_\eta$,
  \be\label{ldpub a23b}
  \left \{ \omega\in \Omega:
 \
u^\eps (\cdot ,0,  y
) \in  \caln  
\left (
I^{s-{\frac 14} s_2}_{r_0 -kt_0, y },\
{\frac 12} s_1
\right )
\right \}
\subseteq
 \left \{ \omega\in \Omega:
 \
u^\eps (r_0 -kt_0, 0,  y
) \in J^s_{s_1}
\right \}.
 \ee
Actually,   for every 
$u^\eps (\cdot, y
) \in  \caln 
\left (
I^{s-{\frac 14}
\delta_2}_{r_0 -kt_0,  y },\
{\frac 12} s_1
\right )$,  
there exists
$\phi \in
I^{s-{\frac 14}
s_2}_{r_0 -kt_0, y }$ such that
$$
\sup_{t\in 
[0, r_0 -kt_0]
 }
 \| u^\eps (t, 0, y 
)  -\phi (t)\|
<{\frac 12} s_1,
$$
and hence by \eqref{ldpub a23a} we get
$$
{\rm dist}_{\ell^2}
\left (
u^\eps (r_0 -kt_0, 0, y)
, \  J^s  \right )
\le
\| u^\eps (r_0 -kt_0, 0, y)
-\phi (r_0 -kt_0)\|
+  
{\rm dist}_{\ell^2}
\left (
\phi  (r_0 -kt_0)
, \  J^s   \right )
< s_1,
$$
and thus  \eqref{ldpub a23b} is valid.
By \eqref{ldpub a23b} we  get
\be\label{ldpub a24}
P\left (
u^\eps (r_0 -kt_0, 0,y )
\notin J^s_{s_1}
\right )
\le
P
 \left (  
u^\eps (\cdot , 0, y
) \notin  \caln  
\left (
I^{s-{\frac 14}
\delta_2}_{r_0 -kt_0, y },\
{\frac 12} s_1
\right )
\right  ).
\ee
 By  \eqref{ldpub a24} and
Theorem  \ref{uldp}
 we see that
there exists $\eps_3\in (0, \eps_2)$
such that for all $\eps \le \eps_3$,
  \be\label{ldpub a25}
\sup_{y\in \cala_\delta}
P\left (
u^\eps (r_0 -kt_0, 0,y)
\notin J^s_{s_1}
\right )
\le
e^{
-{\frac 1\eps}
(s-{\frac 12} s_2)
}.
\ee

It follows from  \eqref{ldpub a23} and
 \eqref{ldpub a25}  
 that
 for all $\eps \le \eps_3$,
  \be\label{ldpub a26}
 \sum_{k=1}^{n_0}
  \int_{\|u_0\| \le R}
P \left (u^\eps (r_0, 0, u_0)
\notin J^s_{s_1},
\  u^\eps (kt_0, 0, u_0 ) \in  \cala_{\eta}
 \right )\ 
\mu^\eps (d u_0)
\le
n_0
e^{
-{\frac 1\eps}
(s-{\frac 12} s_2)
}.
\ee
 By \eqref{ldpub a21},
 \eqref{ldpub a22}
 and
  \eqref{ldpub a26} we 
  obtain that
  for all $\eps \le \eps_3$,
  \be\label{ldpub a27}
 \int_{\|u_0 \| \le R}
P\left 
 (u^\eps (r_0, 0, u_0 )
\notin J^s_{s_1},
  \ u^\eps (\cdot, 0, u_0 ) \notin \caly_{n_0}
 \right ) \ 
\mu^\eps (d u_0 )
\le  
n_0 e^{-
{\frac {s  }
{\eps}
}}
+
 n_0 e^{-
{\frac {s
	 -{\frac 12}s_2  }
{\eps}
}}.
\ee
By \eqref{ldpub a15}
and
\eqref{ldpub a27}
 we get
for all $\eps \le \eps_3$,
$$
\mu^\eps
(\ell^2\setminus
J^s_{s_1})
\le
 (2+n_0)  e^{- {\frac {s}{\eps}}}
  +
 n_0 e^{-
{\frac {s
	 -{\frac 12}s_2  }
{\eps}
}} 
\le
\left (
(2+n_0)  e^{-
{\frac {s_2}
{\eps}
}}
+
 n_0 e^{-
{\frac {s_2}
{2\eps}
}}
\right )
e^{-
{\frac {s-s_2}
{\eps}
}}.
 $$
 Note that
 $(2+n_0)  e^{-
{\frac {s_2}
{\eps}
}}
+
 n_0 e^{-
{\frac {s_2}
{2\eps}
}} \to 0$ as $\eps \to 0$
 and hence 
 \eqref{ldpub a1} is valid.
 In other words, 
 the family $\{\mu^\eps\}_{\eps \in (0,1)}$
 satisfies the   LDP
 upper bound.

   \newpage

      \section{Statements and Declarations} 
      
       \subsection{Funding}
       No funding was received for conducting this study.

  \subsection{Competing Interests} 
   The author has  no relevant financial or non-financial
    interests to disclose.
    
            \subsection{Author contributions}
    I  was the 
contributor in writing the manuscript.

%
%
%

\end{document}